\documentclass[reqno,11pt]{amsart}
\usepackage[latin1]{inputenc}
\usepackage[T1]{fontenc}
\usepackage{lmodern}
\usepackage{a4wide}
\usepackage{amsmath,amsfonts,amsthm,amssymb}
\usepackage{enumitem}
\usepackage{mathtools}
\usepackage{color}
\usepackage{dsfont}

\usepackage{hyperref}

\renewcommand\leq{\leqslant}
\renewcommand\geq{\geqslant}

\newcommand\be{\begin{equation}}
\newcommand\ee{\end{equation}}

\theoremstyle{plain}
\newtheorem{theorem}{\bf Theorem}[section]
\newtheorem{corollary}{\bf Corollary}[section]
\newtheorem{lemma}{\bf Lemma}[section]
\newtheorem{remark}{\bf Remark}[section]
\newtheorem{definition}{\bf Definition}[section]

\numberwithin{equation}{section}

\title[Imbedding results in Musielak-Orlicz spaces...]{Imbedding results in Musielak-Orlicz spaces with
an application to anisotropic nonlinear Neumann
problems}

\author[A. Youssfi]{Ahmed Youssfi}
\address{University Sidi Mohamed Ben Abdellah,
	National School of Applied Sciences,
	P.O. Box 72 F\`es-Pricipale, Fez, Morocco}
\email{address:ahmed.youssfi@gmail.com ; ahmed.youssfi@usmba.ac.ma}

\author[M. M. Ould Khatri]{Mohamed Mahmoud Ould Khatri}
\address{University Sidi Mohamed Ben Abdellah,
	National School of Applied Sciences,
	P.O. Box 72 F\`es-Pricipale, Fez, Morocco}
\email{mahmoud.ouldkhatri@usmba.ac.ma}

\begin{document}
\maketitle

\begin{abstract}
We prove a continuous embedding that allows us to obtain a boundary trace imbedding result for anisotropic Musielak-Orlicz spaces, which we then apply to obtain an existence result for Neumann problems with
nonlinearities on the boundary associated to some anisotropic nonlinear elliptic equations in Musielak-Orlicz spaces constructed from Musielak-Orlicz functions on which and on their conjugates we do not assume the $\Delta_2$-condition. The uniqueness is also studied.
\end{abstract}




\section{Introduction}\label{1}
\par Let $\Omega$ be an open bounded subset of $\mathbb{R}^{N}$, $(N\geq2)$. We denote by
$\vec{\phi}:\Omega\times\mathbb{R}^{+}\to\mathbb{R}^{N}$ the vector function
$\vec{\phi}=(\phi_{1},\cdots,\phi_{N})$
where for every $i\in\{1,\cdots,N\}$, $\phi_{i}$ is a Musielak-Orlicz function derivable with respect to its second argument whose
complementary Musielak-Orlicz
function is denoted by $\phi_{i}^{\ast}$ (see preliminaries).
We consider the follwing problem
\begin{equation}\label{p1}
\left\{
\begin{array}{rll}\displaystyle
-\sum_{i=1}^{N}\partial_{x_{i}}a_{i}(x,\partial_{x_{i}}u)
+b(x)\varphi_{max}
(x,|u(x)|)&=f(x,u)&\mbox{ in }\Omega,\\
u&\geq0&\mbox{ in }\Omega,\\\displaystyle
\sum_{i=1}^{N}a_{i}(x,\partial_{x_{i}}u)\nu_{i}&=g(x,u)&\mbox{ on }\partial\Omega,
\end{array}
\right.
\end{equation}
where $\partial_{x_{i}}=\frac{\partial}{\partial_{x_{i}}}$ and for every $i=1,\cdots,N$, we denote by $\nu_{i}$
the $i^{th}$ component of the outer normal unit vector and $a_{i}:\Omega\times\mathbb{R}\rightarrow\mathbb{R}$ is a Carath\'{e}odory
function such that there exist a locally integrable (see Definition \ref{dfn1.1}) Musielak-Orlicz function $P_{i}$ with $P_{i}\ll\phi_{i}$ (see \eqref{grows}), a positive constant $c_{i}$ and a
nonnegative function $d_{i}\in E_{\phi_{i}^{\ast}}
(\Omega)$ satisfying for all $s$, $t\in\mathbb{R}$ and for almost every $x\in\Omega$ the following assumptions
\begin{equation}\label{a1}
|a_{i}(x,s)|\leq c_{i}[d_{i}(x)+(\phi_{i}^{\ast})^{-1}(x,P_{i}(x,s))],
\end{equation}
\begin{equation}\label{a2}
\phi_{i}(x,|s|)\leq a_{i}(x,s)s\leq A_{i}(x,s),
\end{equation}
\begin{equation}\label{a3}
\big(a_{i}(x,s)-a_{i}(x,t)\big)\cdot\big(s-t\big)>0, \mbox{ for all }s\neq t,
\end{equation}
the function $A_{i}:\Omega\times\mathbb{R}\rightarrow\mathbb{R}$ is defined by
$$
A_{i}(x,s)=\int_{0}^{s}a_{i}(x,t)dt.
$$
Here and in what follows, we define the functions $\phi_{min}$ and $\phi_{max}$ as follows
$$
\phi_{min}(x,s)=\min_{i=1,\cdots,N}\phi_{i}(x,s)\mbox{ and }
\phi_{max}(x,s)=\max_{i=1,\cdots,N}\phi_{i}(x,s).
$$
Let us denote $\varphi_{max}(x,y)=\frac{\partial\phi_{max}}{\partial y}(x,y)$.
We also assume that there exist a locally integrable Musielak-Orlicz function $R$ with $R\ll\phi_{max}$ and a nonnegative function $D\in E_{\phi^{\ast}_{max}}(\Omega)$, such that for all $s$, $t\in\mathbb{R}$ and for almost every $x\in\Omega$
\begin{equation}\label{phi.max1}
|\varphi_{max}(x,s)|\leq D(x)+(\phi_{max}^{\ast})^{-1}(x,R(x,s)),
\end{equation}
where $\phi_{max}^{\ast}$ stands for the complementary function of $\phi_{max}$ defined below in \eqref{compl1}.
\par For what concerns the data, we suppose that   $f:\Omega\times\mathbb{R}\rightarrow\mathbb{R}^+$ and
$g:\partial\Omega\times\mathbb{R}\rightarrow\mathbb{R}$ are  Carath\'{e}odory functions. We define the
antiderivatives
$F:\Omega\times\mathbb{R}\rightarrow\mathbb{R}$ and $G:\partial\Omega\times\mathbb{R}\rightarrow\mathbb{R}$
of $f$ and $g$ respectively by
\begin{equation*}\label{f}
F(x,s)=\int_{0}^{s}f(x,t)dt
\mbox{ and }
G(x,s)=\int_{0}^{s}g(x,t)dt.
\end{equation*}
We say that a Musielak-Orlicz function $\phi$ satisfies the $\Delta_{2}$-condition, if there exists a positive constant $k>0$ and
a nonnegative function $h\in L^{1}(\Omega)$ such that
$$
\phi(x,2t)\leq k\phi(x,t)+h(x).
$$
We remark that the $\Delta_{2}$-condition is equivalent to the following condition:
for all $\alpha>1$ there exists a positive constant $C_{\alpha}>0$ and
a nonnegative function $h_{\alpha}\in L^{1}(\Omega)$ such that
$$
\phi(x,\alpha t)\leq C_{\alpha}\phi(x,t)+h_{\alpha}(x).
$$
We assume now that there exist two positive constants $k_{1}$ and $k_{2}$ and two locally integrable Musielak-Orlicz functions
$M$ and $H$ satisfy the $\Delta_{2}$-condition and derivable with respect to their second arguments with $M\ll\phi_{min}^{\ast\ast}$,
$H\ll\phi_{min}^{\ast\ast}$ and $H\ll\psi_{min}$,
such that the functions $f$ and $g$ satisfy for all $s\in\mathbb{R}_{+}$ the following assumptions
\begin{equation}\label{F}
|f(x,s)|\leq k_{1}m(x,s),\mbox{ for a.e. }x\in\Omega,
\end{equation}
\begin{equation}\label{G}
|g(x,s)|\leq k_{2}
h(x,s),\mbox{ for a.e. }x\in\partial\Omega,
\end{equation}
where
\begin{equation}\label{psi}
\psi_{min}(x,t)=[(\phi_{min}^{\ast\ast})_{\ast}(x,t)]^{\frac{N-1}{N}},\; m(x,s)=\frac{\partial M(x,s)}{\partial s} \mbox{ and }
h(x,s)=\frac{\partial H(x,s)}{\partial s}.
\end{equation}
Finally, for the function $b$ involved in \eqref{p1}, we assume that there exists a constant
$b_{0}>0$ such that $b$ satisfies the hypothesis
\begin{equation}\label{b}
b\in L^{\infty}(\Omega)\mbox{ and } b(x)\geq b_{0},\mbox{ for a.e. }x\in\Omega.
\end{equation}
We remark that \eqref{a3} and the relation $a_{i}(x,\zeta)=\nabla_{\zeta}
A_{i}(x,\zeta)$ imply in particular that for any $i=1,\cdots,N$, the function $\zeta\to A_{i}(\cdot,\zeta)$ is convex.
\par Let us put ourselves in the particular case of $\vec{\phi}=(\phi_{i})_{i\in\{1,\cdots,N\}}$ where for $i\in\{1,\cdots,N\}$, $\phi_{i}(x,t)=|t|^{p_{i}(x)}$ with $p_{i}\in C_{+}(\bar{\Omega})=\{h\in C(\bar{\Omega}):\inf_{x\in\Omega}h(x)>1\}$. Defining $p_{max}(x)=\max_{i\in\{1,\cdots,N\}}p_{i}(x)$ and $p_{min}(x)=\min_{i\in\{1,\cdots,N\}}p_{i}(x)$, one has
$\phi_{max}(x,t)=|t|^{p_{M}(x)}$ and therefore $\varphi_{max}(x,t)=p_{M}(x)|t|^{p_{M}(x)-2}t$,
where $p_{M}$ is $p_{max}$ or $p_{min}$ according to whether $|t|\geq1$ or $|t|\leq1$. Therefore, the problem \eqref{p1}
can be rewritten as follows
\begin{equation}\label{p2}
\left\{
\begin{array}{rll}\displaystyle
-\sum_{i=1}^{N}\partial_{x_{i}}a_{i}(x,\partial_{x_{i}}u)+b_{1}(x)|u|^{p_{M}(x)-2}u
&=f(x,u)&\mbox{ in }\Omega,\\
u&\geq0&\mbox{ in } \Omega,\\\displaystyle
\sum_{i=1}^{N}a_{i}(x,\partial_{x_{i}}u)\nu_{i}&=g(x,u)&\mbox{ on } \partial\Omega,
\end{array}%
\right.
\end{equation}
where $b_{1}(x)=p_{M}(x)b(x)$.
Boureanu and R\v{a}dulescu \cite{Rad} have proved the existence and uniqueness of the weak solution of \eqref{p2}.
They prove an imbedding and a trace results which they use together with a classical minimization existence result for
functional reflexive
framework (see \cite[Theorem 1.2]{Struwe}). Problem \eqref{p2} with Dirichlet boundary condition and
$b_{1}(x)=0$ was treated in \cite{Kone}. The authors proved that if $f(\cdot,u)=f(\cdot)\in L^{\infty}(\Omega)$ then
\eqref{p2} admits a unique solution by using \cite[Theorem 1.2]{Struwe}.\\
Let us mention some related results in the framework of Orlicz-Sobolev spaces. Le and Schmitt \cite{Smitt} proved an
existence result for the following boundary value problem
\begin{equation*}
\left\{
\begin{array}{rll}
-\mbox{div}(A(|\nabla u|^{2})\nabla u)+F(x,u)=0,
&\mbox{ in } \Omega,\\
u=0&\mbox{ on } \partial\Omega,
\end{array}%
\right.
\end{equation*}
in $W_{0}^{1}L_{\phi}(\Omega)$ where $\phi(s)=A(|s|^{2})s$ and $F$ is a Carath\'{e}odory function satisfying some growth conditions. This result extends the one obtained in \cite{Garcia} with
$F(x,u)=-\lambda\psi(u)$, where $\psi$ is an odd increasing homeomorphism of $\mathbb{R}$ onto $\mathbb{R}$. In \cite{Garcia,Smitt} the authors assume that the $N$-function $\overline{\phi}$ complementary to the $N$-function $\phi$ satisfies
the $\Delta_{2}$ condition, which is used to prove that the functional $u\to \int_{\Omega}\Phi(|\nabla u|)dx$ is coercive and of class $\mathcal{C}^1$, where $\Phi$ is the antiderivative of $\phi$ vanishing at origin.
\par Here, we are interested in proving the existence and uniqueness of the weak solution for problem \eqref{p1} without any additional condition on the
Musielak-Orlicz function $\phi_{i}$ or its complementary $\overline{\phi}_{i}$ for $i=1,\cdots,N.$ Thus, the Musielak-Orlicz spaces $L_{\phi_{i}}(\Omega)$
are neither reflexive nor separable and hence classical existence results can not be applied.
\par The approach we use consists in proving first a continuous imbedding and a trace result which we then apply to solve problem \eqref{p1}. The results we prove here extend to the anisotropic Musielak-Orlicz-Sobolev spaces the continuous imbedding result obtained in \cite{Fan} under extra conditions and the trace result proved in \cite{LWZ}. To the best of our knowledge, the trace result we obtain here is new and does not exist in the literature. The main difficulty we found when we deal with problem \eqref{p1} is the coercivity of the energy functional. We overcome this by using both our continuous imbedding and  trace results. Then we prove the boundedness of a minimization sequence and by a compactness argument, we are led to obtain a minimizer which is a weak solution of problem \eqref{p1}.
\begin{definition}\label{dfn1.1}
Let $\Omega$ be an open subset of $\mathbb{R}^{N}$, ($N\geq2$). We say that a Musielak-Orlicz function $\phi$ is locally-integrable, if for every compact subset $K$ of $\Omega$ and every constant $c>0$,
$$
\int_{K}\phi(x,c)dx<\infty.
$$
\end{definition}
The paper is organized as follows : Section \ref{2} contains some definitions. In Section \ref{3},
we give and prove our main  results, which we then apply in Section \ref{4} to solve problem \eqref{p1}. In the last section
we give the appendix which contains some important lemmas that are necessary for the accomplishment of the proofs of the
results.
\section{Preliminaries}\label{2}
\subsection{\textbf{Anisotropic Musielak-Orlicz-Sobolev spaces}}
Let $\Omega$ be an open subset of $\mathbb{R}^{N}$. A real function $\phi:\Omega\times\mathbb{R}^{+}\rightarrow\mathbb{R}^{+}$, will be called a Musielak-Orlicz function, if it satisfies the following conditions:
\begin{itemize}
\item [(i)] $\phi(\cdot,t)$ is a measurable function on $\Omega$.
\item [(ii)] $\phi(x,\cdot)$ is an $N$-function, that is a convex and nondecreasing  function with $\inf_{x\in\Omega}\phi(x,1)>0$, $\phi(x,t)=0$ if only
if $t=0$, $\phi(x,t)>0$ for all $t>0$ and for almost every $x\in \Omega$,
$$
\lim_{t\rightarrow\infty}\frac{\phi(x,t)}{t}=\infty\mbox{ and }
\lim_{t\rightarrow0}\frac{\phi(x,t)}{t}=0.
$$
\end{itemize}
We will extend these Musielak-Orlicz functions into even functions on all $\Omega\times\mathbb{R}$.
The complementary function $\phi^{\ast}$ of the Musilek-Orlicz function $\phi$ is defined by
\begin{equation}\label{compl1}
\phi^{\ast}(x,s)=\sup_{t\geq0}\{st-\phi(x,t)\}.
\end{equation}
It can be checked that $\phi^{\ast}$ is also a Musielak-Orlicz function (see \cite{Mus}). Moreover, for every $t$, $s\geq0$ and a.e.
$x\in\Omega$ we have the so-called Young inequality (see \cite{Mus})
$$
ts\leq\phi(x,t)+\phi^{\ast}(x,s).
$$
\par For any function $h:\mathbb{R}\to\mathbb{R}$ the second complementary function $h^{\ast\ast}=(h^{\ast})^{\ast}$ (cf. \eqref{compl1}), is convex and satisfies
\begin{equation}\label{astast}
h^{\ast\ast}(x)\leq h(x),
\end{equation}
with equality when $h$ is convex. Roughly speaking, $h^{\ast\ast}$ is a convex envelope of $h$, that is the biggest convex function smaller or equal to $h$.
\par Let $\phi$ and $\psi$ be two Musielak-Orlicz functions. We say that $\psi$ grows essentially more slowly than $\phi$, denote $\psi\ll\phi$, if
\begin{equation}\label{grows}
\displaystyle\lim_{t\to+\infty}\frac{\psi(x,t)}{\phi(x,ct)}=0,
\end{equation}
for every constant $c>0$ and for almost every $x\in\Omega$. We remark that for a locally-integrable Musielak function $\psi$, if $\psi\ll\phi$ then for all $c>0$ there exists a nonnegative integrable function $h$, depending on $c$, such that
$$
\psi(x,t)\leq\phi(x,ct)+h(x), \mbox{ for all }t\in\mathbb{R} \mbox{ and for a.e. }x\in\Omega.
$$
The Musielak-Orlicz space $L_{\phi}(\Omega)$ is defined by
$$
L_{\phi}(\Omega)=\Big\{u:\Omega\rightarrow\mathbb{R}\mbox{ measurable }
/\int_{\Omega}\phi\Big(x,\frac{u(x)}{\lambda}\Big)<+\infty\mbox{ for some }
\lambda>0\Big\}.
$$
Endowed with the so-called Luxemborg norm
\begin{equation*}
\|u\|_{\phi}=\inf\Big\{\lambda>0/\int_{\Omega}\phi\Big(x,\frac{u(x)}
{\lambda}\Big)dx\leq1\Big\},
\end{equation*}
$(L_{\phi}(\Omega),\|\cdot\|_{\phi})$ is a Banach space.
Observe that if $\Omega$ is of finite measure, from the fact that $\displaystyle\inf_{x\in\Omega}\phi(x,1)>0$ follows the following continuous imbedding (see \cite[Lemma 1]{Hudzik})
\begin{equation}\label{imbd.l1}
L_{\phi}(\Omega)\hookrightarrow L^{1}(\Omega).
\end{equation}
This will be also the case if instead of $\displaystyle\inf_{x\in\Omega}\phi(x,1)>0$ we assume $\displaystyle\lim_{t\rightarrow\infty}\inf_{x\in\Omega}\frac{\phi(x,t)}{t}=\infty$.
We will also use the space $E_{\phi}(\Omega)$ defined by
$$
E_{\phi}(\Omega)=\Big\{u:\Omega\rightarrow\mathbb{R}\mbox{ measurable }
/\int_{\Omega}\phi\Big(x,\frac{u(x)}{\lambda}\Big)<+\infty\mbox{ for all }
\lambda>0\Big\}.
$$
Observe that for every $u\in L_{\phi}(\Omega)$ the following inequality holds
\begin{equation}\label{norm.modular}
\|u\|_{\phi}\leq\int_{\Omega}\phi(x,u(x))dx+1.
\end{equation}
For two complementary Musielak-Orlicz functions $\phi$ and $\phi^{\ast}$, the following H\"{o}lder's
inequality (see \cite{Mus})
\begin{equation}\label{Holder}
\int_{\Omega}|u(x)v(x)|dx\leq2\|u\|_{\phi}\|v\|_{\phi^{\ast}}
\end{equation}
holds for every $u\in L_{\phi}(\Omega)$ and $v\in L_{\phi^{\ast}}(\Omega)$.
Define $\phi^{\ast-1}$ for every $s\geq0$ by
$$
\phi^{\ast-1}(x,s)=\sup\{\tau\geq0:\phi^{\ast}(x,\tau)\leq s\}.
$$
Then, for almost every $x\in\Omega$ and for every $s\in\mathbb{R}$ we have
\begin{equation}\label{prop1}
\phi^{\ast}(x,\phi^{\ast-1}(x,s))\leq s,
\end{equation}
\begin{equation}\label{prop2}
s\leq\phi^{\ast-1}(x,s)\phi^{-1}(x,s)\leq 2s,
\end{equation}
\begin{equation}\label{prop3}
\phi(x,s)\leq s\frac{\partial\phi(x,s)}{\partial s}\leq\phi(x,2s).
\end{equation}
\begin{definition}
	Let $\vec{\phi}:\Omega\times\mathbb{R}_{+}\longrightarrow\mathbb{R}^{N}$, the vector function
	$\vec{\phi}=(\phi_{1},\cdots,\phi_{N})$
	where for every $i\in\{1,\cdots,N\}$, $\phi_{i}$ is a Musielak-Orlicz function.
	We define the anisotropic Musielak-Orlicz-Sobolev space by
	\begin{equation*}
	W^{1}L_{\vec{\phi}}(\Omega)=\Big\{u\in L_{\phi_{max}}(\Omega);\;
	\partial_{x_{i}}u\in L_{\phi_{i}}(\Omega)\mbox{ for all }i=1,\cdot\cdot\cdot,N\Big\}.
	\end{equation*}
\end{definition}
By the continuous imbedding \eqref{imbd.l1}, we get that $W^{1}L_{\vec{\phi}}(\Omega)$
is a Banach space with respect to the following norm
\begin{equation*}
\|u\|_{W^{1}L_{\vec{\phi}}(\Omega)}=\|u\|_{\phi_{max}}+\sum_{i=1}^{N}\|\partial_{x_{i}}u\|_{\phi_{i}}.
\end{equation*}
Moreover, in view of \cite[Lemma 1]{Hudzik} we have the continuous embedding 
$W^{1}L_{\vec{\phi}}(\Omega)\hookrightarrow W^{1,1}(\Omega)$.
\section{Main results}\label{3}
In this section we prove an imbedding theorem and a trace result. Let us assume the following conditions
\begin{equation}\label{phi.min3}
\int^{1}_{0}\frac{(\phi_{min}^{\ast\ast})^{-1}(x,t)}{t^{1+\frac{1}{N}}}dt<+\infty\mbox{ and }
\int^{+\infty}_{1}\frac{(\phi_{min}^{\ast\ast})^{-1}(x,t)}{t^{1+\frac{1}{N}}}dt
=+\infty,\;\; \forall x\in\overline{\Omega}.
\end{equation}
Thus, we define the Sobolev conjugate $(\phi_{min}^{\ast\ast})_\ast$
of $\phi_{min}^{\ast\ast}$ by
\begin{equation}\label{phi.min3.1}
(\phi_{min}^{\ast\ast})_{\ast}^{-1}(x,s)=\int^{s}_{0}\frac{(\phi_{min}^{\ast\ast})^{-1}(x,t)}
{t^{1+\frac{1}{N}}}dt,\; \mbox{ for } x\in\overline{\Omega} \mbox{ and } s\in [0,+\infty).
\end{equation}
It may readily be checked that $(\phi_{min}^{\ast\ast})_\ast$ is an $N$-function.
We assume that there exist two positive constants $\nu<\frac{1}{N}$ and $c_{0}$, such that
\begin{equation}\label{phi.min4}
\Big|\frac{\partial(\phi_{min}^{\ast\ast})_\ast(x,t)}{\partial x_{i}}\Big|
\leq c_{0}\Big[(\phi_{min}^{\ast\ast})_\ast(x,t)+((\phi_{min}^{\ast\ast})_\ast(x,t))^{1+\nu}\Big],
\end{equation}
for all $t\in\mathbb{R}$ and for almost every $x\in\Omega$, provided that for every $i=1,\cdots,N$ the derivative
$\frac{\partial(\phi_{min}^{\ast\ast})_\ast(x,t)}{\partial x_{i}}$ exists.
\subsection{An imbedding theorem}
\begin{theorem}\label{thm1}
	Let $\Omega$ be an open bounded subset of $\mathbb{R}^{N}$, $(N\geq2)$, with the cone property. Assume that \eqref{phi.min3} and \eqref{phi.min4} are fulfilled, $(\phi_{min}^{\ast\ast})_\ast(\cdot,t)$ is Lipschitz continuous on $\overline{\Omega}$ and $\phi_{max}$ is locally integrable. Then, there is a continuous embedding
	$$
	W^{1}L_{\vec{\phi}}(\Omega)\hookrightarrow L_{(\phi_{min}^{\ast\ast})_\ast}(\Omega).
	$$
\end{theorem}
\begin{remark}\label{remark1}~
	\begin{enumerate}
        \item In the case where for all $i=1,\cdots,N$, $\phi_{i}(x,t)=\phi(x,t)=|t|^{p(x)},$ then $\phi_{min}^{\ast\ast}(x,s)=\phi_{min}(x,t)=\phi(x,t)$ and so
            $(\phi_{min}^{\ast\ast})_{\ast}(x,s)=(\phi_{min})_{\ast}(x,t)=\phi_{\ast}(x,t)$.
		\item If we consider $\phi(x,t)=|t|^{p(x)}$ with $\phi(\cdot,t)$ is lipschitz continuous on $\overline{\Omega}$. Then the function $p(\cdot)$ is Lipschitz continuous on $\overline{\Omega}$. Indeed,
		we have that $\phi(x,\cdot)$ is locally Lipschitz continuous on $\mathbb{R}$ (because $\phi(x,\cdot)$ is convex on $\mathbb{R}$), let $t\in[a,b]$ with $t\neq1$ and $a,b\in\mathbb{R}^{+}$ with $a>0$. We can write
		$$
		|p(x)-p(y)|=\frac{1}{|\log |t||}\Big|\log(|t|^{p(x)})-
		\log(|t|^{p(y)})\Big|.
		$$
		Being the function $\log(\cdot)$ locally Lipschitz continuous on $\mathbb{R}^{+}_{\ast}$, there exists a constant $k>0$ satisfying
		\begin{equation}\label{Proof.remark1}
		|p(x)-p(y)|\leq k||t|^{p(x)}-|t|^{p(y)}|=k|\phi(x,t)-\phi(y,t)|.
		\end{equation}
		So it follows that $p(\cdot)$ is Lipschitz continuous on $\overline{\Omega}$.
		
		\item If $\phi(x,t)=|t|^{p(x)}$, then denoting $p_{\ast}(x)=\frac{Np(x)}{N-p(x)}$ its Sobolev conjugate is given by $\phi_\ast(x,t)=\Big(\frac{1}{p_{\ast}(x)}\Big)^{p_{\ast}(x)}|t|^{p_{\ast}(x)}$, provided that $p(x)\leq p_{+}=\sup_{x\in\overline{\Omega}}p(x)<N$. Therefore,
		if $\phi_\ast(\cdot,t)$ is lipschitz continuous on $\overline{\Omega}$, then so is $p(\cdot)$ and $\phi_{\ast}$ satisfies \eqref{phi.min4}. Indeed, for every $x$, $y\in\overline{\Omega}$ we have
		$$
		|t^{p_{\ast}(x)}-t^{p_{\ast}(y)}|=(p_{\ast}(x))^{p_{\ast}(x)}\Big|\frac{t^{p_{\ast}(x)}}{(p_{\ast}(x))^{p_{\ast}(x)}}-
		\frac{t^{p_{\ast}(y)}}{(p_{\ast}(x))^{p_{\ast}(x)}}\Big|.
		$$
		We can assume without loss of generality that $p_{\ast}(y)\leq p_{\ast}(x)$, so that since
		$(p_{\ast}(x))^{p_{\ast}(x)}\leq\Big(\frac{N^{2}}{N-p_{+}}\Big)^{\frac{N^{2}}{N-p_{+}}}$ one has
		$$
		\begin{array}{lll}\displaystyle
		|t^{p_{\ast}(x)}-t^{p_{\ast}(y)}|\leq\\
		\Big(\frac{N^{2}}{N-p_{+}}\Big)^{\frac{N^{2}}{N-p_{+}}}
		\Big|\frac{t^{p_{\ast}(x)}}{(p_{\ast}(x))^{p_{\ast}(x)}}-
		\frac{t^{p_{\ast}(y)}}{(p_{\ast}(y))^{p_{\ast}(y)}}\Big|=\Big(\frac{N^{2}}{N-p_{+}}\Big)^{\frac{N^{2}}{N-p_{+}}}
		|\phi_{\ast}(x,t)-\phi_{\ast}(y,t)|.
		\end{array}
		$$
		It follows that the function $(\cdot,t)\to t^{p_{\ast}(\cdot)}$ is Lipschitz continuous on $\overline{\Omega}$. Therefore by \eqref{Proof.remark1}, $p_{\ast}(\cdot)$ is Lipschitz continuous on $\overline{\Omega}$. On the other hand, for every $x$, $y\in\overline{\Omega}$ we can assume $p(y)\leq p(x)$. So that we can write
		$$
		|p(x)-p(y)|\leq\frac{N}{N-p(x)}|p(x)-p(y)|\leq \Big|\frac{Np(x)}{N-p(x)}-\frac{Np(y)}{N-p(y)}\Big|.
		$$
		Hence, since $p_{\ast}(\cdot)$ is Lipschitz continuous on $\overline{\Omega}$ then so is $p(\cdot)$. Now we shall prove that $\phi_{\ast}$ satisfies \eqref{phi.min4}. For every $t\in\mathbb{R}$ and for almost every $x\in\Omega$ we have
		$$
		\frac{\partial\phi_{\ast}(x,t)}{\partial x_{i}}=\frac{\partial p_{\ast}(x)}{\partial x_{i}}\log\Big(
		\frac{|t|}{ep_{\ast}(x)} \Big)
		\phi_{\ast}(x,t).
		$$	
		As $p_{\ast}(\cdot)>1$ we get
		$$
		\Big|\frac{\partial\phi_{\ast}(x,t)}{\partial x_{i}}\Big|\leq \Big|\frac{\partial p_{\ast}(x)}{\partial x_{i}}\Big|\Big|\log\Big(
		\frac{|t|}{e} \Big)\Big|
		\phi_{\ast}(x,t)
		$$
		Let $\epsilon>0$. If $\epsilon<\frac{1}{N}$ then for all $t>0$
		we can easily check  that
		$$
		\log(t)\leq \frac{1}{\epsilon^2Ne}t^\epsilon.
		$$
		Now, since the Musielak-Orlicz function $\phi_{\ast}$ has a superlinear growth, we can choose $A>0$ for which there exists $t_0>e$ such that $A|t|\leq \phi_{\ast}(x,t)$ whenever $|t|\geq t_0$. Thus, \\
		$\bullet$ If $|t|\geq t_0$ then
		$$
		\Big|\frac{\partial\phi_{\ast}(x,t)}{\partial x_{i}}\Big|\leq\Big|\frac{\partial p_{\ast}(x)}{\partial x_{i}}\Big|
		\frac{1}{\epsilon^2Ne^{1+\epsilon}}|t|^\epsilon
		\phi_{\ast}(x,t)
		$$
		Since $p_{\ast}(\cdot)$ is Lipschitz continuous on $\overline{\Omega}$ there exists a constant $C>0$ satisfying $\Big|\frac{\partial p_{\ast}(x)}{\partial x_{i}}\Big|\leq C$. Hence,
		$$
		\Big|\frac{\partial\phi_{\ast}(x,t)}{\partial x_{i}}\Big|\leq
		\frac{C}{\epsilon^2Ne^{1+\epsilon}A^\epsilon}
		{\phi_{\ast}}^{1+\epsilon}(x,t).
		$$
		$\bullet$ If $|t|\leq t_0$ then
		$$
		\Big|\frac{\partial\phi_{\ast}(x,t)}{\partial x_{i}}\Big|\leq C|\log\Big(
		\frac{t_0}{e}\Big)\phi_{\ast}(x,t).
		$$
		Therefore, for every $t\in\mathbb{R}$ and for almost every $x\in\Omega$ we always have
		$$
		\frac{\partial\phi_{\ast}(x,t)}{\partial x_{i}}\leq c_{0}\Big[\phi_{\ast}(x,t)+(\phi_{\ast}(x,t))^{1+\epsilon}\Big].
		$$
		\item In the proof of Theorem \ref{thm1} we apply Lemma \ref{lem.impo3} for $g(x,t)=((\phi_{min})_{\ast}(x,t))^{\alpha}$ with $\alpha\in(0,1)$ where the boundedness of $\Omega$ was needed to guarantee that
		$\max_{x\in\overline{\Omega}} g(x,t)<\infty$ for some $t$, which holds automatically in the case where for all $i=1,\cdots,N$, $\phi_{i}(x,t)=\phi(x,t)=|t|^{p(x)}$, where $p:\overline{\Omega}\to\mathbb{R}^+$ is a measurable function such that
		$1\leq p(\cdot)<N$, since
		$\phi_{\ast}(x,t)\leq\min\{1,t^{\frac{N^{2}}{N-p_{+}}}\}$. So in this case we don't need to assume that $\Omega$ is a bounded set and
		the space $W^{1}L_{\vec{\phi}}(\Omega)$ is nothing but the variable exponent Sobolev space $W^{1,p(\cdot)}(\Omega)$. Therefore, the embedding in Theorem \ref{thm1} is an extension to Musielak-Orlicz framework of the one already proved by Fan \cite[Theorem 1.1]{fan}.
	\end{enumerate}
\end{remark}
\begin{proof} \textbf{of Theorem \ref{thm1}.}
	Let $u\in W^{1}L_{\vec{\phi}}(\Omega)$. Assume first that the function $u$ is bounded and $u\neq0$.
	Defining $f(s)=\displaystyle\int_{\Omega}(\phi_{min}^{\ast\ast})_{\ast}\Big(x,\frac{|u(x)|}{s}\Big)dx$, for $s>0$, one has
	$\displaystyle\lim_{s\rightarrow0^{+}}f(s)=+\infty$ and $\displaystyle\lim_{s\rightarrow\infty}f(s)=0$. Since $f$ is continuous on
	$(0,+\infty)$, there exists $\lambda>0$ such that $f(\lambda)=1$. Then
	by the definition of the Luxemburg norm, we get
	\begin{equation}\label{imbd1}
	\|u\|_{(\phi_{min}^{\ast\ast})_{\ast}}\leq \lambda.
	\end{equation}
	On the other hand,
$$
f(\|u\|_{(\phi_{min}^{\ast\ast})_{\ast}})=
	\int_{\Omega}(\phi_{min}^{\ast\ast})_{\ast}\Big(x,\frac{u(x)}
	{\|u\|_{(\phi_{min}^{\ast\ast})_{\ast}}}\Big)dx\leq1=f(\lambda)
$$
and since $f$ is decreasing we get
	\begin{equation}\label{imbd2}
	\lambda\leq\|u\|_{(\phi_{min}^{\ast\ast})_{\ast}}.
	\end{equation}
	So that by \eqref{imbd1} and \eqref{imbd2}, we get $\lambda=\|u\|_{(\phi_{min}^{\ast\ast})_{\ast}}$ and
	\begin{equation}\label{imbd3}
	\int_{\Omega}(\phi_{min}^{\ast\ast})_{\ast}\Big(x,\frac{u(x)}{\lambda}\Big)dx=1.
	\end{equation}
	From \eqref{phi.min3.1} we can easily check that $(\phi_{min}^{\ast\ast})_{\ast}$ satisfies the following differential equation
	$$
	(\phi_{min}^{\ast\ast})^{-1}(x,(\phi_{min}^{\ast\ast})_{\ast}(x,t))\frac{\partial(\phi_{min}^{\ast\ast})_{\ast}}{\partial t}(x,t)
	=((\phi_{min}^{\ast\ast})_{\ast}(x,t))^{\frac{N+1}{N}}.
	$$
	Hence, by \eqref{prop2} we obtain the following inequality
	\begin{equation}\label{imbd4}
	\frac{\partial(\phi_{min}^{\ast\ast})_{\ast}}{\partial t}(x,t)\leq ((\phi_{min}^{\ast\ast})_{\ast}(x,t))
	^{\frac{1}{N}}(\phi_{min}^{\ast\ast})^{\ast-1}(x,(\phi_{min}^{\ast\ast})_{\ast}(x,t)), \mbox{ for a.e. }x\in \Omega.
	\end{equation}
	Let $h$ be the function defined by
	\begin{equation}\label{h}
	h(x)=\Big[(\phi_{min}^{\ast\ast})_{\ast}\Big(x,\frac{u(x)}{\lambda}\Big)\Big]
	^{\frac{N-1}{N}}.
	\end{equation}
	Since $(\phi_{min}^{\ast\ast})_{\ast}(\cdot,t)$ is Lipschitz continuous on $\overline{\Omega}$ and $(\phi_{min}^{\ast\ast})_{\ast}(x,\cdot)$ is locally Lipschitz continuous on $\mathbb{R}$, the function $h$ is Lipschitz continuous on $\overline{\Omega}$. Hence, we can compute using Lemma \ref{lem.impo5} (given in Appendix) for $f=h$, obtaining for a.e. $x\in\Omega$,
	\begin{equation*}\label{imbd5}
	\frac{\partial h(x)}{\partial x_{i}}=\frac{N-1}{N}\Big((\phi_{min}^{\ast\ast})_{\ast}
	\Big(x,\frac{u(x)}{\lambda}\Big)\Big)^{-\frac{1}{N}}
	\Big[\frac{\partial(\phi_{min}^{\ast\ast})_{\ast}}{\partial t}\Big(x,\frac{u(x)}{\lambda}\Big)
	\frac{\partial_{x_{i}}u(x)}{\lambda}+\frac{\partial(\phi_{min}^{\ast\ast})_{\ast}
	}{\partial_{x_{i}}}\Big(x,\frac{u(x)}{\lambda}\Big)\Big].
	\end{equation*}
	Therefore,
	\begin{equation}\label{imbd6}
	\displaystyle\sum_{i=1}^{N}\Big|\frac{\partial h(x)}{\partial x_{i}}\Big|\leq I_{1}+I_{2},
	\mbox{ for a.e. }x\in \Omega,
	\end{equation}
	where we have set
	$$
	I_{1}=\frac{N-1}{N\lambda}\Big((\phi_{min}^{\ast\ast})_{\ast}\Big(x,\frac{u(x)}{\lambda}
	\Big)\Big)
	^{\frac{-1}{N}}
	\frac{\partial(\phi_{min}^{\ast\ast})_{\ast}}{\partial t}\Big(x,\frac{u(x)}{\lambda}\Big)\displaystyle\sum_{i=1}
	^{N}|\partial_{x_{i}}u(x)|
	$$
	and
	$$
	I_{2}=\frac{N-1}{N}\Big((\phi_{min}^{\ast\ast})_{\ast}\Big(x,\frac{u(x)}{\lambda}\Big)
	\Big)
	^{\frac{-1}{N}}\displaystyle\sum_{i=1}^{N}
	\Big|\frac{\partial(\phi_{min}^{\ast\ast})_{\ast}}{\partial x_{i}}\Big(x,\frac{u(x)}{\lambda}\Big)\Big|.
	$$
	We shall now estimate the two integrals $\displaystyle\int_{\Omega}I_{1}(x)dx$ and $\displaystyle\int_{\Omega}I_{2}(x)dx$.
By \eqref{imbd4}, we can write
	\begin{equation}\label{imbd7}
	I_{1}(x)\leq\frac{N-1}{N\lambda}(\phi_{min}^{\ast\ast})^{\ast-1}\Big(x,(\phi_{min}^{\ast\ast})_{\ast}
	\Big(x,\frac{u(x)}{\lambda}\Big)\Big)\displaystyle\sum_{i=1}^{N}|\partial_{x_{i}}u(x)|.
	\end{equation}
By \eqref{prop1}, we have
	\begin{equation*}
    \begin{array}{clc}\displaystyle
    \int_{\Omega}(\phi_{min}^{\ast\ast})^{\ast}\Big(x,(\phi_{min}^{\ast\ast})^{\ast-1}\Big(x,(\phi_{min}^{\ast\ast})_{\ast}
	\Big(x,\frac{u(x)}{\lambda}\Big)\Big)\Big)dx\leq
    \int_{\Omega}(\phi_{min}^{\ast\ast})_{\ast}\Big(x,\frac{u(x)}{\lambda}\Big)dx\leq1.
    \end{array}
	\end{equation*}
	Hence
	\begin{equation}\label{imbd8}
	\Big\|(\phi_{min}^{\ast\ast})^{\ast-1}\Big(x,(\phi_{min}^{\ast\ast})_{\ast}\Big(x,\frac{u(x)}
	{\lambda}\Big)\Big)\Big\|_{(\phi_{min}^{\ast\ast})^{\ast}}\leq1.
	\end{equation}
	It follows from \eqref{Holder}, \eqref{imbd7} and \eqref{imbd8} that
	\begin{equation}\label{imbd9}
	\begin{array}{clc}\displaystyle
	\int_{\Omega}I_{1}(x)dx &\leq \displaystyle\frac{2(N-1)}{N\lambda}
	\Big\|(\phi_{min}^{\ast\ast})^{\ast-1}\Big(x,(\phi_{min}^{\ast\ast})_{\ast}\Big(x,\frac{u(x)}
	{\lambda}\Big)\Big)\Big\|_{(\phi_{min}^{\ast\ast})^{\ast}}\displaystyle\sum_{i=1}^{N}\Big\|\partial_{x_{i}}u(x)\Big\|
	_{\phi_{min}^{\ast\ast}}\\
	&\leq\displaystyle\frac{2(N-1)}{N\lambda}\sum_{i=1}^{N}\Big\|\partial_{x_{i}}u(x)\Big\|_{\phi_{min}^{\ast\ast}}\\
	&\leq\displaystyle\frac{2}{\lambda}\sum_{i=1}^{N}\Big\|\partial_{x_{i}}u(x)\Big\|_{\phi_{min}^{\ast\ast}}.
	\end{array}
	\end{equation}
	Recall that by the definition of $\phi_{min}$ and \eqref{astast}, we get $\|\partial_{x_{i}}u(x)\|_{\phi_{min}^{\ast\ast}}\leq\|\partial_{x_{i}}u(x)\|_{\phi_{i}}$,
	so that \eqref{imbd9} implies
	\begin{equation}\label{imbd10}
	\int_{\Omega}I_{1}(x)dx\leq
	\frac{2}{\lambda}\displaystyle\sum_{i=1}^{N}\Big\|\partial_{x_{i}}u(x)\Big\|_{\phi_{i}}.
	\end{equation}
	By \eqref{phi.min4}, we can write
	$$
	I_{2}(x)\leq c_{1}\Big[\Big((\phi_{min}^{\ast\ast})_{\ast}\Big(x,\frac{u(x)}{\lambda}\Big)\Big)
^{1-\frac{1}{N}}+\Big((\phi_{min}^{\ast\ast})_{\ast}\Big(x,\frac{u(x)}{\lambda}\Big)
	\Big)^{1-\frac{1}{N}+\nu}\Big],
	$$
	with $c_{1}=c_{0}(N-1).$ Since $(\phi_{min}^{\ast\ast})_{\ast}(\cdot,t)$ is continuous on $\overline{\Omega}$ and $\nu<\frac{1}{N}$, we can apply Lemma \ref{lem.impo3} (given in Appendix) with the functions
	$g(x,t)=\frac{((\phi_{min}^{\ast\ast})_{\ast}(x,t))^{1-\frac{1}{N}+\nu}}{t}$ and $f(x,t)
	=\frac{(\phi_{min}^{\ast\ast})_{\ast}(x,t)}{t}$ with $\epsilon=\frac{1}{8c_{1}c_{\ast}}$, obtaining for $t=\frac{|u(x)|}{\lambda}$
	\begin{equation}\label{I.2.1}
	\Big[(\phi_{min}^{\ast\ast})_{\ast}\Big(x,\frac{u(x)}{\lambda}\Big)\Big]^{1-\frac{1}{N}+
		\nu}\leq\frac{1}{8c_{1}c_{\ast}}(\phi_{min}^{\ast\ast})_{\ast}\Big(x,\frac{u(x)}
	{\lambda}\Big)+K_{0}\frac{|u(x)|}{\lambda}.
	\end{equation}
	Using again Lemma \ref{lem.impo3} (given in Appendix) with the functions
	$g(x,t)=\frac{((\phi_{min}^{\ast\ast})_{\ast}(x,t))^{1-\frac{1}{N}}}{t}$ and $f(x,t)
	=\frac{(\phi_{min}^{\ast\ast})_{\ast}(x,t)}{t}$ with $\epsilon=\frac{1}{8c_{1}c_{\ast}}$, we get by substituting $t$ by $\frac{|u(x)|}{\lambda}$
	\begin{equation}\label{I.2.2}
	\Big[(\phi_{min}^{\ast\ast})_{\ast}\Big(x,\frac{u(x)}{\lambda}\Big)\Big]^{1-\frac{1}{N}}
	\leq\frac{1}{8c_{1}c_{\ast}}(\phi_{min}^{\ast\ast})_{\ast}\Big(x,\frac{u(x)}
	{\lambda}\Big)+K_{0}\frac{|u(x)|}{\lambda},
	\end{equation}
	where $c_{\ast}$ is the constant in the embedding
	$W^{1,1}(\Omega)\hookrightarrow L^{\frac{N}{N-1}}(\Omega)$, that is
	\begin{equation}\label{imbd11}
	\|w\|_{L^{\frac{N}{N-1}}(\Omega)}\leq c_{\ast}\|w\|_{W^{1.1}(\Omega)},\mbox{ for all }
	w\in W^{1,1}(\Omega).
	\end{equation}
	By \eqref{I.2.1} and \eqref{I.2.2}, we obtain
	\begin{equation}\label{imbd12}
	\int_{\Omega}I_{2}(x)dx\leq\frac{1}{4c_{\ast}}+\frac{2K_{0}c_{1}}{\lambda}
	\|u\|_{L^{1}(\Omega)}.
	\end{equation}
	Puting together \eqref{imbd10} and \eqref{imbd12} in \eqref{imbd6}, we obtain
	\begin{equation*}
	\begin{array}{clc}
	\displaystyle\sum_{i=1}^{N}\|\partial_{x_{i}}h\|_{L^{1}(\Omega)}&\leq\frac{1}{4c_{\ast}}+\frac{2}
	{\lambda}\displaystyle\sum_{i=1}^{N}\|\partial_{x_{i}}u(x)\|_{\phi_{i}}
	+\frac{2K_{0}c_{1}}{\lambda}\|u\|_{L^{1}(\Omega)}&\\
	&\leq\frac{1}{4c_{\ast}}+\frac{2}
	{\lambda}\displaystyle\sum_{i=1}^{N}\|\partial_{x_{i}}u(x)\|_{\phi_{i}}
	+\frac{2K_{0}c_{1}c_{2}}{\lambda}\|u\|_{\phi_{max}},&
	\end{array}%
	\end{equation*}
	where $c_{2}$ is the constant in the continuous embedding \eqref{imbd.l1}. Then it follows
	\begin{equation}\label{imbd13}
	\displaystyle\sum_{i=1}^{N}\|\partial_{x_{i}}h\|_{L^{1}(\Omega)}\leq\frac{1}{4c_{\ast}}+\frac{c_{3}}
	{\lambda}\|u\|_{W^{1}L_{\vec{\phi}}(\Omega)},
	\end{equation}
	with $c_{3}=\max\{2,2K_{0}c_{1}c_{2}\}$. Now, using again Lemma \ref{lem.impo3} (given in Appendix) for the functions $g(x,t)=\frac{\big[(\phi_{min}^{\ast\ast})_{\ast}(x,t)\big]^{1-\frac{1}{N}}}{t}$
	and $f(x,t)
	=\frac{(\phi_{min}^{\ast\ast})_{\ast}(x,t)}{t}$ with $\epsilon=\frac{1}{4c_{\ast}}$, we obtain for $t=\frac{|u(x)|}{\lambda}$
	$$
	h(x)\leq\frac{1}{4c_{\ast}}(\phi_{min}^{\ast\ast})_{\ast}\Big(x,\frac{u(x)}{\lambda}\Big)
	+K_{0}\frac{|u(x)|}{\lambda},
	$$
	From \eqref{imbd.l1}, we obtain
	\begin{equation}\label{imbd14}
	\|h\|_{L^{1}(\Omega)}\leq\frac{1}{4c_{\ast}}+\frac{K_{0}c_{2}}{\lambda}
	\|u\|_{L_{\phi_{max}}(\Omega)}.
	\end{equation}
	Thus, by \eqref{imbd13} and \eqref{imbd14} we get
	$$
	\|h\|_{W^{1,1}(\Omega)}\leq\frac{1}{2c_{\ast}}
	+\frac{c_{4}}{\lambda}\|u\|_{W^{1}L_{\vec{\phi}}(\Omega)},
	$$
	where $c_{4}=c_{3}+K_{0}c_{2}$, which shows that $h\in W^{1,1}(\Omega)$ and
	which together with \eqref{imbd11} yield
	$$
	\|h\|_{L^{\frac{N}{N-1}}(\Omega)}\leq\frac{1}{2}+
	\frac{c_{4}c_{\ast}}{\lambda}\|u\|_{W^{1}L_{\vec{\phi}}(\Omega)}.
	$$
	Having in mind \eqref{imbd3}, we get $\displaystyle\int_{\Omega}[h(x)]^{\frac{N}{N-1}}dx=\int_{\Omega}(\phi_{min}^{\ast\ast})_{\ast}
	\Big(x,\frac{u(x)}{\lambda}\Big)dx=1$. So that one has
	\begin{equation}\label{imbd15}
	\|u\|_{(\phi_{min}^{\ast\ast})_{\ast}}=\lambda\leq 2c_{4}c_{\ast}\|u\|_{W^{1}L_{\vec{\Phi}}(\Omega)}.
	\end{equation}
	We now extend the estimate \eqref{imbd15} to an arbitrary $u\in W^{1}L_{\vec{\phi}}(\Omega)$.
	Let $T_{n}$, $n>0$, be the truncation function at levels $\pm n$ defined on $\mathbb{R}$ by
	$T_{n}(s)=\min\{n,\max\{s,-n\}\}$. Since $\phi_{max}$ is locally integrable, by \cite[Lemma 8.34.]{Adams} one has $T_{n}(u)\in W^{1}L_{\vec{\phi}}(\Omega)$. So that in view of \eqref{imbd15}
	\begin{equation}\label{imbd16}
	\|T_{n}(u)\|_{(\phi_{min}^{\ast\ast})_{\ast}}\leq 2c_{4}c_{\ast}\|T_{n}(u)\|_{W^{1}L_{\vec{\phi}(\Omega)}}
	\leq 2c_{4}c_{\ast}\|u\|_{W^{1}L_{\vec{\phi}(\Omega)}}.
	\end{equation}
	Let $k_{n}=\|T_{n}(u)\|_{(\phi_{min}^{\ast\ast})_{\ast}}$. Thanks to \eqref{imbd16}, the sequence $\{k_{n}\}_{n=1}^{\infty}$
	is nondecreasing and converges. If we denote $k=\lim_{n\rightarrow\infty}k_{n}$, by Fatou's
	lemma we have
	$$
	\int_{\Omega}(\phi_{min}^{\ast\ast})_{\ast}\Big(x,\frac{|u(x)|}{k}\Big)dx\leq\lim\inf\int_{\Omega}
	(\phi_{min}^{\ast\ast})_{\ast}\Big(x,\frac{|T_{n}(u)|}{k_{n}}\Big)dx\leq1.
	$$
	This implies that $u\in L_{(\phi_{min}^{\ast\ast})_{\ast}}(\Omega)$ and
	$$
	\|u\|_{(\phi_{min}^{\ast\ast})_{\ast}}\leq k=\lim_{n\rightarrow\infty}\|T_{n}(u)\|_{(\phi_{min}^{\ast\ast})_{\ast}}
	\leq 2c_{4}c_{\ast}\|u\|_{W^{1}L_{\vec{\phi}(\Omega)}}.
	$$
	The theorem is then completely proved.
\end{proof}
\begin{corollary}\label{lem.imbd.comp}
	Let $\Omega$ be an open bounded subset of $\mathbb{R}^{N}$, $(N\geq2)$, with the cone property. Assume that \eqref{phi.min3}, \eqref{phi.min4} are fulfilled,
	$(\phi_{min}^{\ast\ast})_{\ast}(\cdot,t)$ is Lipschitz continuous on $\overline{\Omega}$
	and $\phi_{max}$ is locally integrable.
	Let $A$ be a Musielak-Orlicz function where the function $A(\cdot,t)$ is continuous  on $\overline{\Omega}$ and $A\ll(\phi_{min}^{\ast\ast})_{\ast}$.
	Then, the following embedding
	$$
	W^{1}L_{\vec{\phi}}(\Omega)\hookrightarrow L_{A}(\Omega).
	$$
	is compact.
\end{corollary}
\begin{proof}
	Let $\{u_{n}\}$ is a bounded sequence in $W^{1}L_{\vec{\phi}}(\Omega)$.
	By Theorem \ref{thm1}, $\{u_{n}\}$ is bounded in $L_{(\phi_{min}^{\ast\ast})_{\ast}}(\Omega).$
	Since the embedding
	$W^{1}L_{\vec{\phi}}(\Omega)\hookrightarrow W^{1,1}(\Omega)$ is continuous and the imbedding
	$W^{1,1}(\Omega)\hookrightarrow L^{1}(\Omega)$ is compact, we deduce That there exists a subsequence of $\{u_{n}\}$ still
	denoted by $\{u_{n}\}$ which converges in measure in
	$\Omega$. Since $A\ll(\phi_{min}^{\ast\ast})_{\ast}$, by Lemma \ref{lem.impo4} (given in Appendix) the sequence $\{u_{n}\}$
	converges in norm in $L_{A}(\Omega)$.
\end{proof}
\subsection{A trace result}
We prove here a trace result which is a useful tool to prove the coercivity of some energy functionals. Recall that
$\psi_{min}(x,t)=[(\phi_{min}^{\ast\ast})_{\ast}(x,t)]^{\frac{N-1}{N}}$ is a Musielak-Orlicz function. Indeed, we have
$$
\frac{\partial}{\partial t}(\psi_{min})^{-1}(x,t)=\frac{\partial}{\partial t}(\phi_{min}^{\ast\ast})_{\ast}^{-1}\Big(x,t^{\frac{N}{N-1}}\Big).
$$
By \eqref{phi.min3.1}, we get
$$
\frac{\partial}{\partial t}(\psi_{min})^{-1}(x,t)=\frac{N}{N-1}t^{\frac{1}{N-1}}\frac{(\phi_{min}^{\ast\ast})^{-1}\Big(x,t^{\frac{N}{N-1}}\Big)}
{t^{\frac{N}{N-1}+\frac{1}{N-1}}}=\frac{N}{N-1}\frac{(\phi_{min}^{\ast\ast})^{-1}\Big(x,t^{\frac{N}{N-1}}\Big)}
{t^{\frac{N}{N-1}}}.
$$
Being the inverse of a Musielak-Orlicz function, it is clear that $(\phi_{min}^{\ast\ast})^{-1}$ satisfies
$$
\displaystyle\lim_{\tau\to0^{+}}\frac{(\phi_{min}^{\ast\ast})^{-1}(x,\tau)}{\tau}=+\infty
\mbox{ and }
\displaystyle\lim_{\tau\to+\infty}\frac{(\phi_{min}^{\ast\ast})^{-1}(x,\tau)}{\tau}=0.
$$
Moreover, $(\phi_{min}^{\ast\ast})^{-1}(x,\cdot)$ is concave so that if $0<\tau<\sigma$, then
$$
\frac{(\phi_{min}^{\ast\ast})^{-1}(x,\tau)}{(\phi_{min}^{\ast\ast})^{-1}(x,\sigma)}\geq\frac{\tau}{\sigma}.
$$
Hence, if $0<s_{1}<s_{2}$,
$$
\frac{\frac{\partial}{\partial t}(\psi_{min})^{-1}(x,s_{1})}
{\frac{\partial}{\partial t}(\psi_{min})^{-1}(x,s_{2})}
=\frac{(\phi_{min}^{\ast\ast})^{-1}\Big(x,s_{1}^{\frac{N}{N-1}}\Big)}{(\phi_{min}^{\ast\ast})^{-1}\Big(x,s_{2}^{\frac{N}{N-1}}\Big)}
\frac{s_{2}^{\frac{N}{N-1}}}{s_{1}^{\frac{N}{N-1}}}\geq\frac{s_{1}^{\frac{N}{N-1}}}{s_{2}^{\frac{N}{N-1}}}
\frac{s_{2}^{\frac{N}{N-1}}}{s_{1}^{\frac{N}{N-1}}}=1.
$$
It follows that $\frac{\partial}{\partial t}(\psi_{min})^{-1}(x,t)$ is positive and decreases monotonically from $+\infty$ to $0$ as $t$
increases from $0$ to $+\infty$ and thus $\psi_{min}$ is  a Musielak-Orlicz function.
\begin{theorem}\label{thm2}
	Let $\Omega$ be an open bounded subset of $\mathbb{R}^{N}$, $(N\geq2)$, with the cone property. Assume that \eqref{phi.min3}, \eqref{phi.min4} are fulfilled, $(\phi_{min}^{\ast\ast})_{\ast}(\cdot,t)$ is Lipschitz continuous on $\overline{\Omega}$ and $\phi_{max}$ is locally integrable. Let $\psi_{min}$ the Musielak-Orlicz function defined in \eqref{psi}. Then,  the following boundary trace embedding $
	W^{1}L_{\vec{\phi}}(\Omega)\hookrightarrow L_{\psi_{min}}(\partial\Omega)$
	is continuous.
\end{theorem}
\begin{proof}
	Let $u\in W^{1}L_{\vec{\phi}}(\Omega)$. As the embedding $W^{1}L_{\vec{\phi}}(\Omega)\hookrightarrow
	L_{(\phi_{min})_{\ast}}(\Omega)$ is continuous,
	the function $u$ belongs to $L_{(\phi_{min})_{\ast}}(\Omega)$ and then $u$ belongs to $L_{\psi_{min}}(\Omega)$. Clearly
	$W^{1}L_{\vec{\phi}}(\Omega)\hookrightarrow W^{1,1}(\Omega)$ and by the Gagliardo trace theorem (see \cite{gag}) we have the embedding
	$W^{1,1}(\Omega)\hookrightarrow L^{1}(\partial\Omega)$. Hence, we conclude that for all $u\in W^{1}L_{\vec{\phi}}(\Omega)$ there holds $u|_{\partial\Omega}\in L^{1}(\partial\Omega)$. Therefore, for every $u\in W^{1}L_{\vec{\phi}}(\Omega)$ the trace $u|_{\partial\Omega}$ is well defined. Assume first that $u$ is bounded and $u\neq0$. Since
	$(\phi_{min}^{\ast\ast})_{\ast}(\cdot,t)$ is continuous on $\partial\Omega$, the function $u$ belongs to $L_{\psi_{min}}(\partial\Omega)$.
	Let
	$$
	k=\|u\|_{L_{\psi_{min}}(\partial\Omega)}=\inf\Big\{\lambda>0;\;
	\int_{\partial\Omega}\psi_{min}\Big(x,\frac{u(x)}{\lambda}\Big)dx\leq1\Big\}.
	$$
	We have to distinguish the two cases : $k\geq\|u\|_{L_{(\phi_{min}^{\ast\ast})_{\ast}}(\Omega)}$ and $k<\|u\|_{L_{(\phi_{min}^{\ast\ast})_{\ast}}(\Omega)}$.
	Suppose first that \\
	\underline{Case 1} : Assume that
	\begin{equation}\label{trace1}
	k\geq\|u\|_{L_{(\phi_{min}^{\ast\ast})_{\ast}}(\Omega)}.
	\end{equation}
	Going back to \eqref{h}, we can repeat exactly the same lines with $l(x)=\psi_{min}\Big(x,\frac{u(x)}{k}\Big)$ instead of
	the function $h$, obtaining
	\begin{equation}\label{trace2}
	\|l\|_{W^{1,1}(\Omega)}
	\leq \Big[\frac{1}{4c}+\frac{c_{3}}{k}\|u\|_{W^{1}L_{\vec{\phi}}(\Omega)}
	+\|l\|_{L^{1}(\Omega)}\Big],
	\end{equation}
	where $c$ is the constant in the imbedding $W^{1,1}(\Omega)\hookrightarrow L^{1}(\partial\Omega)$, that is
	\begin{equation}\label{trace3}
	\|w\|_{L^{1}(\partial\Omega)}\leq c\|w\|_{W^{1,1}(\Omega)},\mbox{ for all }w\in
	W^{1,1}(\Omega).
	\end{equation}
	Since $(\phi_{min}^{\ast\ast})_{\ast}(\cdot,t)$ is continuous on $\overline{\Omega}$, using Lemma \ref{lem.impo3} (given in Appendix)
with the functions
	$f(x,t)=\frac{(\phi_{min}^{\ast\ast})_{\ast}
		(x,t)}{t}$ and $g(x,t)=\frac{l(x)}{t}$ and $\epsilon=\frac{1}{4c}$, we obtain for $t=\frac{|u(x)|}{k}$
	\begin{equation}\label{trace4}
	l(x)\leq\frac{1}{4c}(\phi_{min}^{\ast\ast})_{\ast}\Big(x,\frac{u(x)}{k}\Big)+K_{0}\frac{|u(x)|}{k}.
	\end{equation}
	By \eqref{trace1}, we have
	$$
	\int_{\Omega}(\phi_{min}^{\ast\ast})_{\ast}\Big(x,\frac{u(x)}{k}\Big)dx
	\leq\int_{\Omega}(\phi_{min}^{\ast\ast})_{\ast}\Big(x,\frac{u(x)}{\|u\|_{(\phi_{min}^{\ast\ast})_{\ast}}}\Big)dx\leq1.
	$$
	Integrating \eqref{trace4} over $\Omega$, we obtain
	\begin{equation}\label{trace5}
	\|l\|_{L^{1}(\Omega)}\leq\frac{1}{4c}+\frac{K_{0}c_{2}}{k}\|u(x)\|_{L_{\phi_{max}}(\Omega)}
	\leq\frac{1}{4c}+\frac{K_{0}c_{2}}{k}\|u\|_{W^{1}L_{\vec{\phi}}(\Omega)},
	\end{equation}
	where $c_{2}$ is the constant of the imbedding \eqref{imbd.l1}. Thus, by virtue of \eqref{trace2} and \eqref{trace5} we get
	$$
	\|l\|_{W^{1,1}(\Omega)}\leq\frac{1}{2c}+\frac{C_{4}}{k}\|u\|_{W^{1}L_{\vec{\phi}}(\Omega)},
	$$
	where $C_{4}=c_{2}K_{0}+c_{3}$.
	This implies that $l\in W^{1,1}(\Omega)$ and by \eqref{trace3} we arrive at
	$$
	\|l\|_{L^{1}(\partial\Omega)}\leq \frac{1}{2}+\frac{cC_{4}}{k}\|u\|_{W^{1}L_{\vec{\phi}}(\Omega)}.
	$$
	As
	$\|l\|_{L^{1}(\partial\Omega)}=\displaystyle\int_{\partial\Omega}|l(x)|dx=
	\displaystyle\int_{\partial\Omega}\psi_{min}\Big(x,\frac{u(x)}{k}\Big)dx=1$, we get
	$$
	\|u\|_{L_{\psi_{min}}(\partial\Omega)}=k\leq 2cC_{4}\|u\|_{W^{1}L_{\vec{\phi}}(\Omega)}.
	$$
	\underline{Case 2} : Assume that
	$$
	k<\|u\|_{(\phi_{min}^{\ast\ast})_{\ast}}.
	$$
	By Theorem \ref{thm1}, there is a constant $c>0$ such that
	$$
	\|u\|_{L_{\psi_{min}}(\partial\Omega)}=k<\|u\|_{(\phi_{min}^{\ast\ast})_{\ast}}\leq
	c\|u\|_{W^{1}L_{\vec{\phi}}(\Omega)}.
	$$
	Finally, in both cases there exists a constant $c>0$ such that
	$$
	\|u\|_{L_{\psi_{min}}(\partial\Omega)}\leq c\|u\|_{W^{1}L_{\vec{\phi}}(\Omega)}.
	$$
	For an arbitrary $u\in W^{1}L_{\vec{\phi}}(\Omega)$, we proceed as in the proof of Theorem \ref{thm1} by
	truncating the function $u$.
\end{proof}
\section{Application to some anisotropic elliptic equations}\label{4}
In this section, we apply the above results to get the existence and the uniqueness results of the weak solution for the problem \eqref{p1}.
\subsection{Properties of the energy functional}
\begin{definition}
	Let $\Omega$ be an open bounded subset of $\mathbb{R}^{N}$, $(N\geq2)$.
	By a weak solution of problem \eqref{p1}, we mean a function $u\in W^{1}L_{\vec{\phi}}(\Omega)$ satisfying for all
    $v\in C^{\infty}(\overline{\Omega})$ the identity
	\begin{equation}\label{dfn}
	\int_{\Omega}\sum_{i=1}^{N}a_{i}(x,\partial_{x_{i}}u)\partial_{x_{i}}vdx
	+\int_{\Omega}b(x)\varphi_{max}(x,u)vdx
	-\int_{\Omega}f(x,u)vdx
	-\int_{\partial\Omega}g(x,u)v ds=0.
	\end{equation}
\end{definition}
We note that all the terms in \eqref{dfn} make sense.
Indeed, for the first term in the right hand side in \eqref{dfn},
we can write by using \eqref{prop3}
$$
\int_{\Omega}\phi^{\ast}_{i}(x,\phi^{\ast-1}_{i}(x,P_{i}(x,\partial_{x_{i}}u(x))))dx
\leq\int_{\Omega}P_{i}(x,\partial_{x_{i}}u(x))dx\leq\int_{\Omega}p_{i}(x,\partial_{x_{i}}u(x))\partial_{x_{i}}u(x)dx,
$$
where $P_{i}$ is the Musielak-Orlicz function given in \eqref{a1} and $p_{i}(x,s)=\frac{\partial P_{i}(x,s)}{\partial s}$. Since $P_{i}$ is locally integrable and $P_{i}\ll\phi_{i}$, we can use Lemma \ref{lem.impo6} (given in Appendix) obtaining $p_{i}(\cdot,\partial_{x_{i}}u(\cdot))\in L_{P_{i}^{\ast}}(\Omega)$. So that by H\"older's inequality \eqref{Holder}, we get
$$
\int_{\Omega}\phi_{i}^{\ast}(x,\phi^{\ast-1}_{i}(x,P_{i}(x,\partial_{x_{i}}u(x))))dx\leq2\|p_{i}(\cdot,\partial_{x_{i}}u(\cdot))
\|_{P_{i}^{\ast}}
\|\partial_{x_{i}}u\|_{P_{i}}<\infty.
$$
Thus, $\phi^{\ast-1}_{i}(\cdot,P_{i}(\cdot,\partial_{x_{i}}u(\cdot)))\in
L_{\phi_{i}^{\ast}}(\Omega)$. Since $v\in C^{\infty}(\overline{\Omega})$ and $\phi_{max}$ is locally integrable, then $v\in W^{1}L_{\vec{\phi}}(\Omega)$.
So we can use the growth condition \eqref{a1} and again the
H\"older inequality \eqref{Holder}, to write
\begin{equation}\label{dfn1}
\int_{\Omega}a_{i}(x,\partial_{x_{i}}u)\partial_{x_{i}}vdx\leq 2c_{i}\|d_{i}(\cdot)\|_{\phi_{i}^{\ast}}
\|\partial_{x_{i}}v\|_{\phi_{i}}+2c_{i}\|\phi^{\ast-1}_{i}(\cdot,P_{i}(\cdot,\partial_{x_{i}}u(\cdot)))\|
_{\phi_{i}^{\ast}}\|\partial_{x_{i}}v\|_{\phi_{i}}<\infty.
\end{equation}
For the second term, the inequality \eqref{prop3} enables us to write
$$
\int_{\Omega}\phi_{max}^{\ast}(x,\phi_{max}^{\ast-1}(x,R(x,u(x))))dx
\leq\int_{\Omega}R(x,u(x))dx\leq\int_{\Omega}r(x,u(x))u(x)dx,
$$
where $R$ is the Musielak-Orlicz function given in \eqref{phi.max1} and $r(x,s)=\frac{\partial R(x,s)}{\partial s}$.
Since $R$ is locally integrable and $R\ll\phi_{max}$, Lemma \ref{lem.impo6} (given in Appendix) gives
$$
\int_{\Omega}\phi_{max}^{\ast}(x,\phi_{max}^{\ast-1}(x,R(x,\partial_{i}u)))dx\leq2\|r(\cdot,u(\cdot))\|_{R^{\ast}}
\|u\|_{R}<\infty,
$$
which shows that $\varphi_{max}(\cdot,u(\cdot))\in L_{\phi_{max}^{\ast}}(\Omega)$. Thus,
\begin{equation}\label{dfn2}
\int_{\Omega}b(x)\varphi_{max}(x,u)vdx\leq2\|b\|_{\infty}\|\varphi_{max}(\cdot,u(\cdot))\|_{\phi_{max}^{\ast}}
\|v\|_{\phi_{max}}<\infty.
\end{equation}
We now turn to the third term in the right hand side in \eqref{dfn}. By using \eqref{F} and the H\"older inequality \eqref{Holder}, one has
\begin{equation}\label{dfn3}
\int_{\Omega}f(x,u)vdx\leq k_{1}\|m(\cdot,u(\cdot))\|_{L_{M^{\ast}}(\Omega)}
\|v\|_{L_{M}(\Omega)}.
\end{equation}
Since $M$ is locally integrable and $M\ll\phi_{min}^{\ast\ast}$, then $M\ll\phi_{max}$ and Lemma \ref{lem.impo6} (given in Appendix)
ensures that
$\int_{\Omega}f(x,u)vdx<\infty.$
For the last term in the right hand side in \eqref{dfn}, using \eqref{G} to have
$$
\int_{\partial\Omega}g(x,u)vds\leq k_{2}\int_{\partial\Omega}h(x,u)vds.
$$
Since the primitive $H$ of $h$ is a locally integrable function satisfies $H\ll\phi_{min}^{\ast\ast}$, thus we can
us a similar way as in Lemma \ref{lem.impo6} (given in Appendix) to get that $h(x,u)\in L_{H^{\ast}}(\partial\Omega)$ and since $\partial\Omega$
is a bounded set, then the imbedding \eqref{imbd.l1} gives that $h(x,u)\in L^{1}(\partial\Omega)$. On the other hand, since
$v\in C^{\infty}(\overline{\Omega})$, then $v\in L^{\infty}(\partial\Omega)$. Therefor,
$$
\int_{\partial\Omega}g(x,u)vds<\infty.
$$
Define the functional $I:W^{1}L_{\vec{\phi}}(\Omega)\to\mathbb{R}$ by
\begin{equation}\label{I}
I(u)=\int_{\Omega}\sum_{i=1}^{N}A_{i}(x,\partial_{i}u)dx+
\int_{\Omega}b(x)\phi_{max}(x,u)dx
-\int_{\Omega}F(x,u)dx-\int_{\partial\Omega}G(x,u)ds.
\end{equation}
\par Some basic properties of $I$ are established in the following lemma.
\begin{lemma}\label{lem.I1}
	Let $\Omega$ be an open bounded subset of $\mathbb{R}^{N}$, $(N\geq2)$.
	Then
	\par(i) The functional $I$ is well defined on $W^{1}L_{\vec{\phi}}(\Omega)$.
	\par(ii) The functional $I$ has a G\^ateaux derivative $I^{\prime}(u)$ for every $u\in W^{1}L_{\vec{\phi}}(\Omega)$. Moreover, for every $v\in W^{1}L_{\vec{\phi}}(\Omega)$
	\begin{equation*}
	\begin{array}{lll}\displaystyle
	\langle I^{\prime}(u),v\rangle&=\displaystyle\int_{\Omega}\sum_{i=1}^{N}a_{i}(x,\partial_{i}u)
	\partial_{i}vdx+\int_{\Omega}b(x)\varphi_{max}(x,u)vdx\\
	&\;\;\;-\displaystyle\int_{\Omega}f(x,u)vdx-\int_{\partial\Omega}g(x,u)vds.
	\end{array}
	\end{equation*}
	So that, the critical points of $I$ are weak solutions to problem \eqref{p1}.
\end{lemma}
\begin{proof}
	(i) For almost every $x\in\Omega$ and for every $\zeta\in\mathbb{R}$, we can write
	$$
	A_{i}(x,\zeta)=\int_{0}^{1}\frac{d}{dt}A_{i}(x,t\zeta)dt=\int_{0}^{1}
	a_{i}(x,t\zeta)\zeta dt.
	$$
	Then, by \eqref{a1} we get
	$$
	A_{i}(x,\zeta)\leq c_{i}d_{i}(x)\zeta+\int_{0}^{1}\phi^{\ast-1}_{i}(x,P_{i}(x,t\zeta))\zeta dt
	\leq c_{i}d_{i}(x)\zeta+\phi^{\ast-1}_{i}(x,P_{i}(x,\zeta))\zeta.
	$$
	In a similar way as in \eqref{dfn1}, we arrive at
	$$
	\Big|\int_{\Omega}A_{i}(x,\partial_{i}u(x))dx\Big|<\infty.
	$$
	Hence, the first term in the right hand side in \eqref{I} is well defined. For the second term, using \eqref{prop3}, the
	H\"older inequality \eqref{Holder} and \eqref{dfn2}, we obtain
	$$\Big|\int_{\Omega}b(x)\phi_{max}(x,u(x))dx\Big|\leq2\|b\|_{\infty}\|\varphi_{max}(\cdot,u(\cdot))\|_{\phi_{max}^{\ast}}\|u\|_{\phi_{max}}
	<\infty,
	$$
	while for the third term, by \eqref{F} and \eqref{prop3} we can write
	$$
	\int_{\Omega}\big|F(x,u(x))\big|dx\leq k_{1}\int_{\Omega}m(x,u)udx.
	$$
	Then, using the H\"older inequality \eqref{Holder} together with \eqref{dfn3} we get
	$$
	\int_{\Omega}\big|F(x,u(x))\big|dx\leq 2k_{1}\|m(\cdot,u(\cdot))\|_{M^{\ast}}\|u\|_{M}<\infty.
	$$
	For the last term in the right hand side in \eqref{I}, we can use \eqref{G} and \eqref{prop3} to have
    $$
    \int_{\partial\Omega}G(x,u)ds\leq k_{2}\int_{\partial\Omega}H(x,u)ds\leq\int_{\partial\Omega}h(x,u)uds.
    $$
    Since $H$ is a locally integrable function satisfies $H\ll\phi_{min}^{\ast\ast}$, then $H\ll\phi_{max}$
    and we can us a similar method as
    in Lemma \ref{lem.impo6} (given in Appendix) to get that $h(x,u)\in L_{H^{\ast}}(\partial\Omega)$. Thus, the H\"older inequality \eqref{Holder} implies that
    $$
    \int_{\partial\Omega}G(x,u)ds<\infty.
    $$
\\
	(ii) For every $i=1,\cdots,N$ define the
	functional $\Lambda_{i}:W^{1}L_{\vec{\phi}}(\Omega)\to\mathbb{R}$ by
	$$
	\Lambda_{i}(u)=\int_{\Omega}A_{i}(x,\partial_{i}u(x))dx.
	$$
	Denote by $B$, $L_{1}$, $L_{2}:W^{1}L_{\vec{\phi}}(\Omega)\to\mathbb{R}$ the functionals $B(u)=\int_{\Omega}b(x)\phi_{max}(x,u(x))dx$
	$L_{1}(u)=\int_{\Omega}F(x,u(x))dx$  and $L_{2}(u)=\int_{\partial\Omega}G(x,u(x))ds$.
	We observe that for $u\in W^{1}L_{\vec{\phi}}(\Omega)$, $v\in C^{\infty}(\overline{\Omega})$ and $r>0$
	$$
	\frac{1}{r}[\Lambda_{i}(u+rv)-\Lambda_{i}(u)]
	=\int_{\Omega}
	\frac{1}{r}\Big[A_{i}\Big(x,\frac{\partial}{\partial x_{i}}u(x)+r
	\frac{\partial}{\partial x_{i}}v(x)\Big)-A_{i}\Big(x,\frac{\partial}{\partial x_{i}}u(x)
	\Big)\Big]dx
	$$
	and
	$$
	\frac{1}{r}\Big[A_{i}\Big(x,\frac{\partial}{\partial x_{i}}u(x)
	+r\frac{\partial}{\partial x_{i}}v(x)\Big)-A_{i}\Big(x,\frac{\partial}
	{\partial x_{i}}u(x)\Big)\Big]
	\longrightarrow a_{i}\Big(x,\frac{\partial}{\partial x_{i}}u(x)
	\Big)\frac{\partial}{\partial x_{i}}v(x),
	$$
	as $r\rightarrow0$ for almost every $x\in\Omega$. On the other hand, the mean value theorem, there exists  $\nu\in[0,1]$ such that
	\begin{equation*}
	\begin{array}{lll}\displaystyle
	\frac{1}{r}\Big|A_{i}\Big(x,\frac{\partial}{\partial x_{i}}u(x)+r
	\frac{\partial}{\partial x_{i}}v(x)\Big)-A_{i}\Big(x,\frac{\partial}{\partial x_{i}}u(x)
	\Big)\Big|\\
	=\displaystyle\Big|a_{i}\Big(x,\frac{\partial}{\partial x_{i}}u(x)+\nu r
	\frac{\partial}{\partial x_{i}}v(x)\Big)\Big|\Big|\frac{\partial}{\partial x_{i}}v(x)
	\Big|.
	\end{array}
	\end{equation*}
	Hence, by using the last equality and \eqref{a1} we get
	\begin{equation*}
	\begin{array}{lll}\displaystyle
	\frac{1}{r}\Big|A_{i}\Big(x,\frac{\partial}{\partial x_{i}}u(x)+r
	\frac{\partial}{\partial x_{i}}v(x)\Big)-A_{i}\Big(x,\frac{\partial}{\partial x_{i}}u(x)
	\Big)\Big|\leq \\\displaystyle
	c_{i}\Big[d_{i}(x)+\phi^{\ast-1}_{i}\Big(x,P_{i}\Big(x,
	\frac{\partial}{\partial x_{i}}u(x)+\nu r
	\frac{\partial}{\partial x_{i}}v(x)\Big)\Big)\Big]
	\Big|\frac{\partial}{\partial x_{i}}v(x)\Big|.
	\end{array}
	\end{equation*}
	Next, by the H\"older inequality \eqref{Holder} we get
	$$
	c_{i}\Big[d_{i}(x)+\phi^{\ast-1}_{i}\Big(x,P_{i}\Big(x,
	\frac{\partial}{\partial x_{i}}u(x)+\nu r
	\frac{\partial}{\partial x_{i}}v(x)\Big)\Big)\Big]\Big|\frac{\partial}{\partial x_{i}}v(x)\Big|\in L^{1}(\Omega).
	$$
	The dominated convergence can be applied to yield
	$$
	\lim_{r\rightarrow0}\frac{1}{r}[\Lambda_{i}(u+rv)-\Lambda_{i}(u)]
	=\int_{\Omega}a_{i}\Big(x,\frac{\partial}{\partial x_{i}}u(x)\Big)
	\frac{\partial}{\partial x_{i}}v(x)dx:=\langle\Lambda^{\prime}_{i}(u),v\rangle,
	$$
	for every $i=1,\cdots,N$. By a similar calculus as in above, we can show that
	$\langle B^{\prime}(u),v\rangle=\int_{\Omega}b(x)\varphi_{max}(x,u)vdx$,
	$\langle L_{1}^{\prime}(u),v\rangle=\int_{\Omega}f(x,u)vdx$  and
	$\langle L_{2}^{\prime}(u),v\rangle=\int_{\Omega}g(x,u)vdx$.
\end{proof}
\subsection{An existence result}
Our main existence result is the following.
\begin{theorem}\label{thm3}
	Let $\Omega$ be an open bounded subset of $\mathbb{R}^{N}$, $(N\geq2)$, with the cone property. Assume that \eqref{a1}, \eqref{a2}, \eqref{a3}, \eqref{F}, \eqref{G}, \eqref{b}, \eqref{phi.min3} and \eqref{phi.min4} are fulfilled and suppose that $\phi_{max}$ and $\phi_{min}^{\ast}$ are locally integrable and $(\phi_{min}^{\ast\ast})_{\ast}(\cdot,t)$ is Lipschitz continuous on $\overline{\Omega}$. Then, problem \eqref{p1} admits at least a weak solution in $W^{1}L_{\vec{\phi}}(\Omega)$.
\end{theorem}
\begin{proof} We divide the proof into three steps.\\
	\textbf{Step 1 }: Weak$^\ast$ lower semicontinuity property of $I$. Define the functional
	$J:W^{1}L_{\vec{\phi}}(\Omega)\to\mathbb{R}$ by
	$$
	J(u)=\int_{\Omega}\sum_{i=1}^{N}A_{i}(x,\partial_{i}u)dx+\int_{\Omega}b(x)\phi_{max}
	(x,u)dx,
	$$
	so that
	$$
	I(u)=J(u)-L_{1}(u)-L_{2}(u).
	$$
	First, we claim that $J$ is sequentially weakly lower semicontinuous. Indeed, since $u\mapsto\phi_{max}(x,u)$ is
	continuous, it is enough to show that the functional
	$$
    u\mapsto K(u)=\displaystyle\int_{\Omega}\displaystyle\sum_{i=1}^{N}A_{i}(x,\partial_{i}u)dx,
    $$
	is sequentially weakly$^\ast$ lower semicontinuous.
	To do this, let $u_{n}\overset{\ast}\rightharpoonup u$ in $W^{1}L_{\vec{\phi}}(\Omega)$ in the sense
	\begin{equation}\label{weak.conv}
	\int_{\Omega}u_{n}\varphi dx\rightarrow\int_{\Omega}u\varphi dx\mbox{ for all }
	\varphi\in E_{\phi_{max}^{\ast}}\mbox{ and }\int_{\Omega}\partial_{i}u_{n}\varphi dx
	\rightarrow\int_{\Omega}\partial_{i}u\varphi dx,
	\end{equation}
	for all $\varphi\in E_{\phi_{i}^{\ast}}$. By the definition of $\phi_{min}$ and $\phi_{max}$, \eqref{weak.conv} holds true for every $\varphi\in E_{\phi_{min}^{\ast}}(\Omega)$.
	Being $\phi_{min}^{\ast}$ locally integrable, one has $L^{\infty}(\Omega)\subset
	E_{\phi_{min}^{\ast}}(\Omega)$. Therefore, for every $i\in\{1,\cdots,N\}$
	\begin{equation}\label{weak.conv1}
	\partial_{i}u_{n}\rightharpoonup \partial_{i}u\mbox{ and }u_{n}\rightharpoonup u
	\mbox{ in }L^{1}(\Omega)\mbox{ for } \sigma(L^{1},L^{\infty}).
	\end{equation}
	As the embedding $W^{1}L_{\vec{\phi}}(\Omega)\hookrightarrow W^{1,1}(\Omega)$ is continuous, the compact embedding
	$W^{1,1}(\Omega)\hookrightarrow L^{1}(\Omega)$ implies that the sequence $\{u_{n}\}$ is relatively compact in $L^{1}(\Omega)$. Therefore, there exist a subsequence still indexed by $n$ and a function $v\in L^{1}(\Omega)$, such that $u_{n}\rightarrow v$ strongly in $L^{1}(\Omega).$
	In view of \eqref{weak.conv1}, we have $v=u$ almost everywhere on $\Omega$ and $u_{n}\rightarrow u$ in $L^{1}(\Omega).$
	Passing once more to a subsequence, we can have $u_{n}\rightarrow u$ almost everywhere on
	$\Omega$. Since $\zeta\rightarrow A_{i}(x,\zeta)$ is convex, by \eqref{a2} we can use
	\cite[Theorem 2.1, Chapter 8]{Ekeland} obtaining
	$$
	K(u)=\int_{\Omega}\sum_{i=1}^{N}A_{i}(x,\partial_{i}u)dx
	\leq\lim\inf\int_{\Omega}\sum_{i=1}^{N}A_{i}(x,\partial_{i}u_{n})dx=\liminf K(u_{n}).
	$$
	We now prove that $L_{1}(u)=\displaystyle\int_{\Omega}F(x,u)dx$ and $L_{2}(u)=\displaystyle\int_{\partial\Omega}G(x,u)ds$ are continuous.
	Since $M\ll\phi_{min}^{\ast\ast}$, it follows that $M\ll(\phi_{min}^{\ast\ast})_{\ast}$, then by Corollary \ref{lem.imbd.comp}
	we get $u_{n}\rightarrow u$ in $L_{M}(\Omega)$.
	Thus, there exists $n_{0}$ such that for every $n\geq n_{0}$, $\|u_{n}-u\|_{M}<\frac{1}{2}$.
	By \eqref{F}, we get
	$$
	\int_{\Omega}|F(x,u_{n}(x))|dx\leq k_{1}\int_{\Omega}M(x,u_{n}(x))dx.
	$$
	Let $\theta_{n}=\|u_{n}-u\|_{M}$. By the convexity of $M$, we can write
	\begin{equation*}
	\begin{array}{clc}
	M(x,u_{n}(x))&=M\Big(x,\theta_{n}\Big(\frac{u_{n}(x)-u(x)}{\theta_{n}}\Big)+(1-\theta_{n})
	\frac{u(x)}{1-\theta_{n}}\Big)&\\
	&\leq\theta_{n}M\Big(x,\frac{u_{n}(x)-u(x)}{\theta_{n}}\Big)+(1-\theta_{n})
	M\Big(x,\frac{u(x)}{1-\theta_{n}}\Big).&
	\end{array}
	\end{equation*}
	Hence,
	\begin{equation}\label{sci1}
	\int_{\Omega}M(x,u_{n}(x))dx\leq\theta_{n}+(1-\theta_{n})
	\int_{\Omega}M\Big(x,\frac{u(x)}{1-\theta_{n}}\Big)dx.
	\end{equation}
	Moreover,
	$$
	M\Big(x,\frac{u(x)}{1-\theta_{n}}\Big)\leq M(x,2u(x)).
	$$
	Since $M$ is locally integrable and $M\ll\phi_{max}$, there exists a nonnegative function
	$h\in L^1(\Omega)$, such that
	$$
	\int_{\Omega}M(x,2|u(x)|)dx\leq\int_{\Omega}\phi_{max}\Big(x,\frac{|u(x)|}{\|u\|_{\phi_{max}}}\Big)dx+\int_{\Omega}h(x)dx
	<\infty.
	$$
	Thus, the Lebesgue dominated convergence theorem yields
	$$
	\lim_{n\rightarrow\infty}\int_{\Omega}M\Big(x,\frac{u(x)}{1-\theta_{n}}\Big)dx=\int_{\Omega}M(x,u(x))dx
	$$
	and therefore, by \eqref{sci1}, we have
	$$
	\limsup_{n\rightarrow\infty}\int_{\Omega}M(x,u_{n}(x))dx\leq\int_{\Omega}M(x,u(x))dx.
	$$
	In addition, by Fatou's Lemma we get
	$$
	\int_{\Omega}M(x,u(x))dx\leq\liminf_{n\rightarrow\infty}\int_{\Omega}M(x,u_{n}(x))dx.
	$$
	Therefore, we have proved that
	$$
	\lim_{n\to+\infty}\int_{\Omega}M(x,u_{n}(x))dx=\int_{\Omega}M(x,u(x))dx.
	$$
	Thus, by \cite[Theorem 13.47]{Hewitt}, we get that $M(x,u_{n}(x))\rightarrow M(x,u(x))$ strongly in $L^{1}(\Omega)$
	which implies that $M(x,u_{n}(x))$ is equi-integrable, then so is $F(x,u_{n}(x))$
	and since $F(x,u_{n})\rightarrow F(x,u)$ almost everywhere on $\Omega$, then by Vitali's theorem we get
	$L_{1}(u_{n})\rightarrow L_{1}(u)$. Similarly, we can show that
	$L_{2}(u_{n})\rightarrow L_{2}(u)$. Thus, $L_{1}$ and $L_{2}$ are continuous
	and since $J$ is weakly$^\ast$ lower semicontinuous, we conclude that $I$ is weakly$^\ast$ lower semicontinuous.\\
	\textbf{Step 2 }: Coercivity of the functional $I$. By \eqref{a2}, \eqref{b} and \eqref{norm.modular}, we can write
	\begin{equation*}
	\begin{array}{lll}\displaystyle
	I(u)&\geq\displaystyle\int_{\Omega}\sum_{i=1}^{N}\phi_{i}(x,\partial_{i}u)dx+b_{0}\int_{\Omega}\phi_{max}(x,u)dx
	-\int_{\Omega}F(x,u)dx-\int_{\partial\Omega}G(x,u)ds&\\
	&\geq\displaystyle\sum_{i=1}^{N}\|\partial_{i}u\|_{\phi_{i}}+b_{0}\|u\|_{\phi_{max}}-N-b_{0}
	-\int_{\Omega}F(x,u)dx-\int_{\partial\Omega}G(x,u)ds&\\
	&\geq\displaystyle\min\{1,b_{0}\}\|u\|_{W^{1}L_{\vec{\phi}}(\Omega)}-\int_{\Omega}F(x,u)dx-\int_{\partial\Omega}G(x,u)ds-N-b_{0}.
	\end{array}
	\end{equation*}
	By \eqref{F} and \eqref{G}, we get
	$$
	I(u)\geq\min\{1,b_{0}\}\|u\|_{W^{1}L_{\vec{\phi}}(\Omega)}-k_{1}\int_{\Omega}M(x,u)dx
	-k_{2}\int_{\partial\Omega}H(x,u)ds-N-b_{0}.
	$$
	As $M\ll(\phi_{min}^{\ast\ast})_{\ast}$ and $H\ll\psi_{min}$, by Theorem \ref{thm1} and Theorem \ref{thm2} there exist two positive constant $C_{1}>0$ and $C_{2}>0$ such that $\|u\|_{L_{M}(\Omega)}\leq C_{1}
	\|u\|_{W^{1}L_{\vec{\phi}}(\Omega)}$ and $\|u\|_{L_{H}(\partial\Omega)}\leq C_{2}
	\|u\|_{W^{1}L_{\vec{\phi}}(\Omega)}$.
	Since $M$ and $H$ satisfy the $\Delta_{2}$-condition, there exist two positive constants $r_{1}>0$ and $r_{2}>0$ and
	two nonnegative functions $h_{1}\in L^1(\Omega)$ and $h_{2}\in L^1(\partial\Omega)$ such that
	\begin{equation*}
	\begin{array}{lll}\displaystyle
	I(u)&\geq\displaystyle\min\{1,b_{0}\}\|u\|_{W^{1}L_{\vec{\phi}}(\Omega)}-k_{1}r_{1}
	\int_{\Omega}M\Big(x,\frac{|u(x)|}{C_{1}\|u\|_{W^{1}L_{\vec{\phi}}(\Omega)}}\Big)dx&\\
	&-\displaystyle k_{2}r_{2}
	\int_{\partial\Omega}H\Big(x,\frac{|u(x)|}{C_{2}\|u\|_{W^{1}L_{\vec{\phi}}(\Omega)}}\Big)ds
	-\int_{\Omega}h_{1}(x)dx-\int_{\partial\Omega}h_{2}(x)ds
	-N-b_{0}&\\
	&\geq\displaystyle\min\{1,b_{0}\}\|u\|_{W^{1}L_{\vec{\phi}}(\Omega)}-k_{1}r_{1}
	\int_{\Omega}M\Big(x,\frac{|u(x)|}{\|u\|_{L_{M}(\Omega)}}\Big)dx&\\
	&\displaystyle-k_{2}r_{2}
	\int_{\partial\Omega}H\Big(x,\frac{|u(x)|}{\|u\|_{L_{H}(\partial\Omega)}}\Big)ds
	-\int_{\Omega}h_{1}(x)dx-\int_{\partial\Omega}h_{2}(x)ds-N-b_{0}.&\\
	&\geq\displaystyle\min\{1,b_{0}\}\|u\|_{W^{1}L_{\vec{\phi}}(\Omega)}
	-\int_{\Omega}h_{1}(x)dx-\int_{\partial\Omega}h_{2}(x)ds-k_{1}r_{1}-k_{2}r_{2}-N-b_{0},&
	\end{array}%
	\end{equation*}
	which implies
	\begin{equation}\label{coer}
	I(u)\rightarrow\infty\mbox{ as }\|u\|_{W^{1}L_{\vec{\phi}}(\Omega)}
	\rightarrow\infty.
	\end{equation}
	\textbf{Step 3 }: Existence of a weak solution. Since $I$ is coercive,
	for an arbitrary $\lambda>0$ there exists $R>0$ such that
	$$
	\|u\|_{W^{1}L_{\vec{\phi}}(\Omega)}>R\Rightarrow I(u)>\lambda.
	$$
	Let $E_{\lambda}=\{u\in W^{1}L_{\vec{\phi}}(\Omega) : I(u)\leq\lambda\}$ and denote by $B_{R}(0)$ the closed ball in $W^{1}L_{\vec{\phi}}(\Omega)$ of radius $R$ centered at origin.
	We claim that $\alpha=\inf_{v\in W^{1}L_{\vec{\phi}}(\Omega)}I(v)>-\infty$. If not, for all $n>0$ there is a sequence $u_n\in E_{\lambda}$ such that $I(u_n)<-n$. As $E_{\lambda}\subset B_{R}(0)$, by the Banach-Alaoglu-Bourbaki theorem there exists $u\in B_{R}(0)$ such that, passing to a subsequence if necessary, we can assume that
	$u_n\rightharpoonup u$ weak$^\ast$ in $W^{1}L_{\vec{\phi}}(\Omega)$. So that the weak$^\ast$ lower semicontinuity of $I$ implies $I(u)=-\infty$ which contradicts the fact that $I$ is well defined on $W^{1}L_{\vec{\phi}}(\Omega)$. Therefore, for every $n>0$ there exists a sequence $u_n\in E_{\lambda}$ such that $I(u_n)\leq \alpha + \frac{1}{n}$. Thus, there exists $u\in B_{R}(0)$ such that, for a subsequence still indexed by $n$, $u_n\rightharpoonup u$ weak$^\ast$ in $W^{1}L_{\vec{\phi}}(\Omega)$. Since $I$ is weakly$^\ast$ lower semicontinuous, we get
	$$
	I(u)=J(u)-L_{1}(u)-L_{2}(u)\leq\liminf_{n\rightarrow\infty} \Big(J(u_{n})-L_{1}(u_{n})-L_{2}(u_{n})\Big)=\liminf_{n\rightarrow\infty}I(u_n)\leq\alpha.
	$$
	Note that $u$ belongs also to $E_{\lambda}$, which yields $I(u)=\alpha\leq\lambda$. This shows that $I(u)=\min\{I(v):v\in W^{1}L_{\vec{\phi}}(\Omega)\}$. Moreover, inserting $v=-u^-$ as test function in \eqref{dfn} and then using \eqref{prop3}, we obtain $u\geq 0$.
	The theorem is completely proved.
\end{proof}
\subsection{Uniqueness result}
In order to prove the uniqueness of the weak solution we have found, we need to assume the following monotony assumptions
\begin{equation}\label{uneq1}
\big(f(x,s)-f(x,t)\big)\big(s-t\big)<0\mbox{ for a.e. }x\in\Omega\mbox{ and for all }s,t
\in\mathbb{R}\mbox{ with } s\neq t
\end{equation}
\begin{equation}\label{uneq2}
\big(g(x,s)-g(x,t)\big)\big(s-t\big)<0\mbox{ for a.e. }x\in\Omega\mbox{ and for all }s,t
\in\mathbb{R}\mbox{ with } s\neq t
\end{equation}
\begin{equation}\label{uneq3}
\big(\varphi_{max}(x,s)-\varphi_{max}(x,t)\big)
\big(s-t\big)>0\mbox{ for a.e. }x\in\Omega\mbox{ and for all }s,t\in\mathbb{R}\mbox{ with } s\neq t.
\end{equation}
\begin{theorem}\label{thm4}
	If in addition to the hypothesis \eqref{a3} the conditions \eqref{uneq1}, \eqref{uneq2} and \eqref{uneq3} are fulfilled, then the weak solution $u$ to problem \eqref{p1} is unique.
\end{theorem}
\begin{proof}
	Suppose that there exists another solution $w$.
	We choose $v=u-w$ as test function in \eqref{dfn} obtaining
	$$
	\begin{array}{lll}\displaystyle
	\int_{\Omega}\sum_{i=1}^{N}a_{i}(x,\partial_{x_{i}}u)\partial_{x_{i}}(u-w)
	dx + \int_{\Omega}b(x)
	\varphi_{max}(x,u)(u-w) dx\\
	\displaystyle-\int_{\Omega}f(x,u)(u-w) dx
	-\int_{\partial\Omega}g(x,u)(u-w) ds=0.
	\end{array}
	$$
	We replace  $u$ by $w$ in \eqref{dfn} and we take $v=w-u$. We obtain
	$$
	\begin{array}{lll}\displaystyle
	\int_{\Omega}\sum_{i=1}^{N}a_{i}(x,\partial_{x_{i}}w)\partial_{x_{i}}(w-u) dx
	+\int_{\Omega}b(x)
	\varphi_{max}(x,w)(w-u) dx\\
	\displaystyle-\int_{\Omega}f(x,w)(w-u) dx
	-\int_{\partial\Omega}g(x,w)(w-u) ds=0.
	\end{array}
	$$
	By combining the previous two equalities, we obtain \\
	\begin{equation*}
	\begin{array}{lll}\displaystyle
	\int_{\Omega}\sum_{i=1}^{N}\Big[a_{i}(x,\partial_{x_{i}}u)-a_{i}(x,\partial_{x_{i}}w)\Big](\partial_{x_{i}}u-\partial_{x_{i}}w)dx\\
	\displaystyle+\int_{\Omega}b(x)\Big[\varphi_{max}(x,u)
	-\varphi_{max}(x,w)\Big](u-w)
	dx\\
	\displaystyle-\int_{\Omega}\Big[f(x,u)-f(x,w)\Big](u-w)dx
	-\int_{\partial\Omega}\Big[g(x,u)-g(x,w)\Big](u-w)ds=0.
	\end{array}
	\end{equation*}
	In view of \eqref{a3}, \eqref{uneq1}, \eqref{uneq2} and \eqref{uneq3}, we obtain $u=w$ a.e. in $\Omega$.
\end{proof}
We recall here some important lemmas that are necessary for the accomplishment of the proofs of the above results.
\begin{lemma}\cite[Theorem 7.10.]{Mus}
Let $\Omega$ be an open bounded subset of $\mathbb{R}^{N}$, and let $\phi$ be a locally integrable Musielak-Orlicz function. Then $E_{\phi}$ is a separable space.
\end{lemma}
\begin{lemma}\label{lem.impo1}
Let $\Omega$ be an open bounded subset of $\mathbb{R}^{N}$, and let $\phi$ be a locally integrable Musielak-Orlicz function. For every $\eta\in L_{\phi^{\ast}}(\Omega)$, the linear functional $F_{\eta}$ defined for every $\zeta\in E_{\phi}(\Omega)$ by
\begin{equation}\label{form.linear}
F_{\eta}(\zeta)=\int_{\Omega}\zeta(x)\eta(x)dx
\end{equation}
belongs to the dual space of $E_{\phi}(\Omega)$, denoted $E_{\phi}(\Omega)^{\ast}$, and its norm $\|F_{\eta}\|$ satisfies
	\begin{equation}\label{norm.form.linear}
	\|F_{\eta}\|\leq2\|\eta\|_{\phi^{\ast}},
	\end{equation}
	where $\|F_{\eta}\|=\sup\{|F_{\eta}(u)|, \|u\|_{L_M(\Omega)}\leq 1\}$.
\end{lemma}
\begin{lemma}\label{lem.impo2}
	Let $\Omega$ be an open bounded subset of $\mathbb{R}^{N}$ and let $\phi$ be a locally integrable Musielak-Orlicz
	function. Then,the dual space of $E_{\phi}(\Omega)$ can be identified to the Musielak-Orlicz space $L_{\phi^{\ast}}(\Omega)$.
\end{lemma}
\begin{proof}
	According to Lemma \ref{lem.impo1} any element $\eta\in L_{\phi^{\ast}}(\Omega)$ defines a bounded linear functional
	$F_{\eta}$ on $L_{\phi}(\Omega)$ and also on $E_{\phi}(\Omega)$ which is given by \eqref{form.linear}. It remains to
	show that every bounded linear functional on $E_{\phi}(\Omega)$ is of the form $F_{\eta}$ for some $\eta\in L_{\phi^{\ast}}(\Omega)$.
	Given $F\in E_{\phi}(\Omega)^{\ast}$, we define the complex measure $\lambda$ by setting
	$$
	\lambda(E)=F(\chi_{E}),
	$$
	where $E$ is a measurable subset of $\Omega$ having finite measure and $\chi_{E}$ stands for the characteristic function of $E$. Due to the fact that $\phi$ is locally integrable, the measurable function
	$\phi\Big(\cdot,\phi^{-1}\Big(x_{0},\frac{1}{2|E|}\Big)\chi_{E}(\cdot)\Big)$ belongs to $L^{1}(\Omega)$ for any $x_{0}\in\Omega$. Hence, there is an
	$\epsilon>0$ such that for any measurable subset $\Omega^{\prime}$ of $\Omega$, one has
	$$
	|\Omega^{\prime}|<\epsilon\Rightarrow\int_{\Omega^{\prime}}\phi\Big(x,\phi^{-1}\Big(x_{0},\frac{1}{2|E|}\Big)\chi_{E}(x)\Big)
	dx\leq\frac{1}{2}.
	$$
	As $\phi(\cdot,s)$ is measurable on $E$, Luzin's theorem implies that for $\epsilon>0$ there exists a closed set
	$K_{\epsilon}\subset E$, with $|E\setminus K_{\epsilon}|<\epsilon$, such that the restriction of $\phi(\cdot,s)$ to
	$K_{\epsilon}$ is continuous. Let $k$ be the point where the maximum of $\phi(\cdot,s)$ is reached in the set $K_{\epsilon}$.
	$$
	\int_{E}\phi\Big(x,\phi^{-1}\Big(k,\frac{1}{2|E|}\Big)\Big)dx=
	\int_{K_{\epsilon}}\phi\Big(x,\phi^{-1}\Big(k,\frac{1}{2|E|}\Big)\Big)dx+
	\int_{E\setminus K_{\epsilon}}\phi\Big(x,\phi^{-1}\Big(k,\frac{1}{2|E|}\Big)\Big).
	$$
	For the first term in the right hand side of the equality, we can write
	$$
	\int_{K_{\epsilon}}\phi\Big(x,\phi^{-1}\Big(k,\frac{1}{2|E|}\Big)\Big)dx\leq
	\int_{E}\phi\Big(k,\phi^{-1}\Big(k,\frac{1}{2|E|}\Big)\Big)dx\leq\frac{1}{2}.
	$$
	Since $|E\setminus K_{\epsilon}|<\epsilon$, the second term can be estimated as
	$$
	\int_{E\setminus K_{\epsilon}}\phi\Big(x,\phi^{-1}\Big(k,\frac{1}{2|E|}\Big)\Big)
	\leq\frac{1}{2}.
	$$
	Thus, we get
	\begin{equation}\label{ineq1.dual}
	\int_{\Omega}\phi\Big(x,\phi^{-1}\Big(k,\frac{1}{2|E|}\Big)\chi_{E}(x)\Big)dx\leq1.
	\end{equation}
	Therefore, we
	$$
	|\lambda(E)|\leq\|F\|\|\chi_{E}\|_{\phi}\leq\frac{\|F\|}{\phi^{-1}\Big(k,\frac{1}{2|E|}\Big)}.
	$$
	As the right-hand side tends to zero when $|E|$ converges to zero, the measure $\lambda$ is absolutely continuous with respect to the Lebesgue measure and so by Radon-Nikodym's Theorem (see for instance \cite[Theorem 1.52]{Adams}), it can be expressed in the form
	$$
	\lambda(E)=\int_{E}\eta(x)dx,
	$$
	for some nonnegative function $\eta\in L^{1}(\Omega)$ unique up to sets of Lebesgue measure zero. Thus,
	$$
	F(\zeta)=\int_{\Omega}\zeta(x)\eta(x)dx
	$$
	holds for every measurable simple function $\zeta$. Note first that since $\Omega$ is bounded and $\phi$ is locally integrable, any measurable simple function lies in $E_{\phi}(\Omega)$ and the set of measurable simple functions is dense in $(E_{\phi}(\Omega),\|\cdot\|_{\phi})$.
	Indeed, for nonnegative $\zeta\in E_{\phi}(\Omega)$, there exists a sequence of increasing measurable simple functions $\zeta_{j}$  converging almost everywhere to $\zeta$ and $|\zeta_{j}(x)|\leq|\zeta(x)|$ on $\Omega$. By the theorem of dominated convergence one has $\zeta_{j}\to \zeta$ in $E_{\phi}(\Omega)$. For an arbitrary $\zeta\in E_{\phi}(\Omega)$, we obtain the same result splitting $\zeta$ into positive and negative parts.\\
	Let $\zeta\in E_{\phi}(\Omega)$ and let $\zeta_{j}$ be a sequence of measurable simple functions converging to $\zeta$ in $E_{\phi}(\Omega)$. By Fatou's Lemma and the inequality \eqref{norm.form.linear} we can write
	$$
	\begin{array}{lll}\displaystyle
	\big|\int_\Omega \zeta(x)\eta(x)dx\big|&\leq\displaystyle\liminf_{j\to+\infty}\int_\Omega |\zeta_{j}(x)\eta(x)|dx=\liminf_{j\to+\infty}F(|\zeta_j|\mbox{sgn}\eta)\\
	&\leq2\displaystyle\|\eta\|_{\phi^{\ast}}\liminf_{j\to+\infty}\|\zeta_j\|_{\phi}\leq
	2\|\eta\|_{\phi^{\ast}}\|\zeta\|_{\phi},
	\end{array}
	$$
	which implies that $\eta\in L_{\phi^{\ast}}(\Omega)$. Thus, the linear functional $F_{\eta}(\zeta)=\displaystyle\int_{\Omega}\zeta(x)\eta(x)dx$ and $F$ defined both on $E_{\phi}(\Omega)$ have the same values on the set of measurable simple functions, so they agree on $E_{\phi}(\Omega)$
	by a density argument.
\end{proof}
\begin{lemma}\label{lem.impo3}
	Let $\Omega$ be an open bounded subset of $\mathbb{R}^{N}$. Let $f,g:\Omega\times(0,+\infty)\to(0,+\infty)$ be continuous nondecreasing functions with respect to there second argument and $g(\cdot,t)$ is continuous on $\overline{\Omega}$ with
	$\displaystyle\lim_{t\rightarrow\infty}\frac{f(x,t)}{g(x,t)}=+\infty$,
	then for all $\epsilon>0$, there exists a positive constant $K_{0}$ such that
	$$
	g(x,t)\leq\epsilon f(x,t)+K_{0}, \mbox{ for all } t>0.
	$$
\end{lemma}
\begin{proof}
	Let $\epsilon>0$ be arbitrary. There exists $t_{0}>0$ such that $t\geq t_{0}$ implies $g(x,t)\leq\epsilon f(x,t)$. Then, for all $t\geq 0$,
	$$
	g(x,t)\leq\epsilon f(x,t)+K(x),
	$$
	where $K(x)=\sup_{t\in(0,t_{0})}g(x,t)$. Being $g(\cdot,t)$ continuous on $\overline{\Omega}$, one has $g(x,t)\leq\epsilon f(x,t)+K_{0}$ with
	$K_{0}=\max_{x\in\overline{\Omega}}K(x)$.
\end{proof}
\begin{lemma}\label{lem.impo4}
	Let $\Omega$ be an open bounded subset of $\mathbb{R}^{N}$.
	Let $A$, $B$ be two Musielak-Orlicz functions such that $B\ll A$, with $B(\cdot,t)$ is continuous on $\overline{\Omega}$. If a sequence
	$\{u_{n}\}$ is bounded in $L_{A}(\Omega)$ and converges in measure in $\Omega,$ then it converges in norm in $L_{B}(\Omega)$.
\end{lemma}
\begin{proof}
	Fix $\epsilon>0$. Defining $v_{j,k}(x)=\frac{u_{j}(x)-u_{k}(x)}{\epsilon}$, we shall show that $\{u_{j}\}$ is a Cauchy sequence in the Banach space $L_{B}(\Omega)$. Clearly $\{v_{j,k}\}$ is
	bounded in $L_{A}(\Omega)$, say $\|v_{j,k}\|_{A}\leq K$ for all $j$ and $k$. Since $B\ll A$ there exists a positive number $t_{0}$ such that for $t\geq t_{0}$ one has
	$$
	B(x,t)\leq \frac{1}{4}A\Big(x,\frac{t}{K}\Big).
	$$
	On the other hand, since $B(\cdot,t)$ is continuous on $\overline{\Omega}$. Let $x_{0}$ be the point where the maximum of
	$B(\cdot,t)$ is reached in $\overline{\Omega}$. Let $\delta=\frac{1}{4B(x_{0},t_{0})}$ and set
	$$
	\Omega_{j,k}=\Big\{x\in\Omega:|v_{j,k}|\geq B^{-1}\Big(x_{0},\frac{1}{2|\Omega|}\Big)\Big\}.
	$$
	Since $\{u_{j}\}$ converges in measure, there exists an integer $N_{0}$ such that $|\Omega_{j,k}|\leq\delta$ whenever
	$j$, $k\geq N_{0}$. Defining
	$$
	\Omega_{j,k}^{\prime}=\{x\in\Omega_{j,k}:|v_{j,k}|\geq t_{0}\}\mbox{ and }
	\Omega_{j,k}^{\prime\prime}=\Omega_{j,k}\setminus\Omega_{j,k}^{\prime},
	$$
	one has
	\begin{equation}\label{lem.impo4.1}
	\begin{array}{lll}\displaystyle
	 \int_{\Omega}B(x,|v_{j,k}(x)|)dx&=\displaystyle\int_{\Omega\setminus\Omega_{j,k}}B(x,|v_{j,k}(x)|)dx+\int_{\Omega_{j,k}^{\prime}}
B(x,|v_{j,k}(x)|)dx\\
	&\;\;\;+\displaystyle\int_{\Omega_{j,k}^{\prime\prime}}B(x,|v_{j,k}(x)|)dx.
	\end{array}
	\end{equation}
	For the first term in the right hand side of \eqref{lem.impo4.1}, we can write
	\begin{equation*}
	\begin{array}{clc}\displaystyle
	\int_{\Omega\setminus\Omega_{j,k}}B(x,|v_{j,k}(x)|)dx&
	\leq\displaystyle\int_{\Omega\setminus\Omega_{j,k}}
	B\Big(x,B^{-1}\Big(x_{0},\frac{1}{2|\Omega|}\Big)\Big)dx\\
	&\leq\displaystyle\int_{\Omega}B\Big(x_{0},B^{-1}\Big(x_{0},\frac{1}{2|\Omega|}\Big)\Big)dx\\
	&\leq\frac{1}{2}.
	\end{array}%
	\end{equation*}
	Since $B\ll A$, the second term in the right hand side of \eqref{lem.impo4.1} can be estimated as follows
	$$
	\int_{\Omega_{j,k}^{\prime}}B(x,|v_{j,k}(x)|)dx\leq\frac{1}{4}\int_{\Omega}A\Big(x,\frac{|v_{j,k}|}{K}\Big)dx\leq\frac{1}{4},
	$$
	while for the third term in the right hand side of \eqref{lem.impo4.1}, we get
	$$
	\int_{\Omega_{j,k}^{\prime\prime}}B(x,|v_{j,k}(x)|)dx\leq\int_{\Omega_{j,k}}B(x,t_{0})dx\leq\delta B(x_{0},t_{0})
	=\frac{1}{4}.
	$$
	Finally, puting all the above estimates in \eqref{lem.impo4.1}, we arrive at
	$$
	\int_{\Omega}B(x,|v_{j,k}(x)|)dx\leq1, \mbox{ for every }j, k\geq N_{0},
	$$
	which yields $\|u_{j}-u_{k}\|_{B}\leq\epsilon$. Thus, $\{u_{j}\}$ converges in the Banach space $L_{B}(\Omega)$.
\end{proof}
\begin{lemma}\label{lem.impo5}
	Let $u\in W_{loc}^{1.1}(\Omega)$ and let $F:\overline{\Omega}\times\mathbb{R}^{+}\to\mathbb{R}^{+}$ be a Lipschitz continuous function. If $f(x)=F(x,u(x))$ then $f\in W_{loc}^{1.1}(\Omega)$. Moreover, for every $j=1,\cdots,N$, the weak derivative $\partial_{x_{j}}f$ of $f$ is such that
	$$
	\partial_{x_{j}}f(x)=\frac{\partial F(x,u(x))}{\partial x_{j}}
	+\frac{\partial F(x,u(x))}{\partial t}\partial_{x_{j}}u(x),\mbox{ for a.e. }x\in \Omega.
	$$
\end{lemma}
\begin{proof}
	Let $\varphi\in D(\Omega)$ and let $\{e_{j}\}_{j=1}^{N}$ be the standard basis in $\mathbb{R}^{N}$. We can write
	\begin{equation*}
	\begin{array}{clc}\displaystyle
	&-\displaystyle\int_{\Omega}F(x,u(x))\partial_{x_{j}}\varphi(x)dx\\
	&=-\displaystyle\lim_{h\to0}\int_{\Omega}F(x,u(x))\frac{\varphi(x)-\varphi(x-he_{j})}{h}dx\\
	&=\displaystyle\lim_{h\to0}\int_{\Omega}\frac{F(x+he_{j},u(x+he_{j}))-F(x,u(x))}{h}\varphi(x)dx\\
	&=\displaystyle\lim_{h\to0}\int_{\Omega}\frac{F(x+he_{j},u(x+he_{j}))-F(x,u(x+he_{j}))}{h}\varphi(x)dx\\
	&\;\;\;+\displaystyle\lim_{h\to0}\int_{\Omega}\frac{F(x,u(x+he_{j}))-F(x,u(x))}{h}\varphi(x)dx\\
	&=\displaystyle\lim_{h\to0}\int_{\Omega}Q_{1}(x,h)\varphi(x)dx\\
	&\;\;\;+\displaystyle\lim_{h\to0}\int_{\Omega}Q_{2}(x,h)\frac{u(x+he_{j})-u(x)}{h}\varphi(x)dx,
	\end{array}
	\end{equation*}
	where
	\begin{equation*}
	Q_{1}(x,h)=\left\{
	\begin{array}{clc}\displaystyle
	\frac{F(x+he_{j},u(x+he_{j}))-F(x,u(x+he_{j}))}{h}&\mbox{ if }h\neq0,&\\\displaystyle
	\frac{\partial F(x,u(x))}{\partial x_{j}}&\mbox{ if }h=0&
	\end{array}
	\right.
	\end{equation*}
	and
	\begin{equation*}
	Q_{2}(x,h)=\left\{
	\begin{array}{clc}\displaystyle
	\frac{F(x,u(x+he_{j}))-F(x,u(x))}{u(x+he_{j})-u(x)}&\mbox{ if }u(x+he_{j})\neq u(x),&\\
	\displaystyle\frac{\partial F(x,u(x))}{\partial t}&\mbox{ otherwise}.
	\end{array}
	\right.
	\end{equation*}
	Since $F(\cdot,\cdot)$ is Lipschitz continuous, there exist two constants $k_{1}$ and $k_{2}>0$ independent of $h$, such that
	$\|Q_{1}(\cdot,h)\|_{\infty}\leq k_{1}$ and $\|Q_{2}(\cdot,h)\|_{\infty}\leq k_{2}$. Thus, for some sequence of values of $h$ tending to zero, $Q_{1}(\cdot,h)$ converges to
	$\frac{\partial F(x,u(x))}{\partial x_{j}}$
	and
	$Q_{2}(\cdot,h)$ converges to $\frac{\partial F(x,u(x))}{\partial t}$
	both in $L^{\infty}(\Omega)$ for the weak-star topology $\sigma^\ast(L^{\infty}(\Omega),L^{1}(\Omega))$.
	On the other hand, since $u\in W^{1,1}(supp(\varphi))$ we have
	$$
	\lim_{h\to0}\int_{supp(\varphi)}\frac{u(x+he_{j})-u(x)}{h}\varphi(x)dx=\int_{supp(\varphi)}\partial_{j}u(x)\varphi(x)dx.
	$$
	It follows that
	$$
	-\int_{\Omega}F(x,u(x))\partial_{x_{j}}\varphi(x)dx=\int_{\Omega}\frac{\partial F(x,u(x))}{\partial x_{j}}\varphi(x)dx+\int_{\Omega}\frac{\partial F(x,u(x))}{\partial t}\partial_{x_{j}}u(x)\varphi(x)dx,
	$$
	which completes the proof of the lemma.
\end{proof}
\begin{lemma}\label{lem.impo6}
Let $\Omega$ be an open subset of $\mathbb{R}^{N}$. Let $A$ and $\phi$ be two Musielak-Orlicz functions with $\phi$ is locally integrable, derivable with respect to its second argument and $\phi\ll A$. Then, $\varphi(\cdot,s)\in L_{\phi^{\ast}}(\Omega)$ for every $s\in L_{A}(\Omega)$, where
$\varphi(x,s)=\frac{\partial\phi(x,s)}{\partial s}$.
\end{lemma}
\begin{proof} Let $s\in L_{A}(\Omega)$.
	By \eqref{prop3}, we can write
$$
\begin{array}{lll}\displaystyle
\int_{\Omega}\phi^{\ast}(x,\varphi(x,s))dx
&=\displaystyle\int_{\Omega}\int_{0}^{\varphi(x,s)}
\varphi^{-1}(x,\tau)d\tau dx\\
&\leq\displaystyle\int_{\Omega}|s|\varphi(x,|s|)dx\\
&\leq\displaystyle\int_{\Omega}\phi(x,2|s|)dx.
\end{array}
$$
It's obvious that if $s=0$ then $\varphi(\cdot,s)\in L_{\phi^{\ast}}(\Omega)$. Assume that $s\neq0$. Since $\phi$ is
locally integrable and $\phi\ll A$, there exists a nonnegative function $h\in L^{1}(\Omega)$ such that
	$\phi(x,2|s|)\leq A\Big(x,\frac{|s|}{\|s\|_{A}}\Big)+h(x)$ for a.e. $x\in\Omega$.
	Thus,
	$$
	\int_{\Omega}\phi(x,2|s|)dx\leq\int_{\Omega}A\Big(x,\frac{|s|}{\|s\|_{A}}\Big)dx+\int_{\Omega}h(x)dx<\infty.
	$$
	Hence, $\varphi(\cdot,s)\in L_{\phi^{\ast}}(\Omega)$.
\end{proof}


\end{document}